\documentclass[a4paper, 11pt]{article}
\usepackage{graphicx} 
\usepackage{amsthm}
\newtheorem{lemma}{Lemma}
\newtheorem{theorem}{Theorem}
\newtheorem{remark}{Remark}
\usepackage{amssymb}
\usepackage{amsmath}
\usepackage{comment}

\usepackage[pagebackref]{hyperref}
\usepackage{lineno}
\usepackage{subcaption}
\usepackage{overpic}
\usepackage{array}
\usepackage{xcolor}

\usepackage[numbers]{natbib}

\newcommand{\Fr}{\textrm{Fr}} 
\newcommand{\Ma}{\textrm{M}\alpha} 

\newcolumntype{P}[1]{>{\centering\arraybackslash}m{#1}}

\makeatother

\title{Primitive variable regularization to derive novel \\Hyperbolic Shallow Water Moment Equations}
\author{
Julian Koellermeier\footnote{Corresponding author, email address {j.koellermeier@rug.nl}} \footnote{Bernoulli Institute, University of Groningen},
\footnote{Department of Mathematics, Computer Science and Statistics, Ghent University}
}

\begin{document}
\maketitle
\begin{abstract}
\noindent
    Shallow Water Moment Equations are reduced-order models for free-surface flows that employ a vertical velocity expansion and derive additional so-called moment equations for the expansion coefficients. Among desirable analytical properties for such systems of equations are hyperbolicity, accuracy, correct momentum equation, and interpretable steady states.
    In this paper, we show analytically that existing models fail at different of these properties and we derive new models overcoming the disadvantages. This is made possible by performing a hyperbolic regularization not in the convective variables (as done in the existing models) but in the primitive variables. Via analytical transformations between the convective and primitive system, we can prove hyperbolicity and compute analytical steady states of the new models. Simulating a dam-break test case, we demonstrate the accuracy of the new models and show that it is essential for accuracy to preserve the momentum equation.
\end{abstract}

\noindent
{\bf Key words:} Free-surface flows; Shallow Water Moment Equations; hyperbolicity; steady states; regularization.


\section{Introduction}
\label{sec:intro}
Free-surface flows are relevant in applications such as tsunami forecasting, river estuary modeling, and flooding risk assessment \cite{Bi2015,Macias2020,Xu2024}. Fast models are essential for acute hazard prediction, while long-term risk assessment requires accuracy. Model reduction serves as a tool to combine these needs by producing models that enable real-time predictions without compromising precision. Although the incompressible Navier–Stokes equations capture the physics of free-surface flows, they are often too computationally expensive for practical use. The Shallow Water Equations (SWE) provide a reduced model obtained by depth-averaging the incompressible Navier–Stokes equations, assuming a constant velocity profile in the vertical direction. However, in many scenarios the assumption of a constant velocity profile is insufficient and may lead to incorrect predictions of bottom friction and other physical effects such as vegetation influence, sediment transport, or bed movement \cite{Garres-Diaz2021a,Huai2021}.

Recently, more refined so-called shallow water moment models have been developed starting with \cite{Kowalski2019}. These models are based on two main steps: (1) Firstly, the incompressible Navier–Stokes equations are reformulated using a vertically rescaled variable $\zeta = (z - h_b)/h$, where $h$ is the water depth from the bottom $h_b$ to the surface, mapping the vertical domain to the unit interval $[0,1]$, also sometimes referred to as $\sigma$-coordinates. This yields a vertically resolved reference system, as detailed in \cite{Kowalski2019}. (2) Secondly, the velocity is expanded in an orthogonal Legendre basis,
\begin{equation}
    u(t, x, \zeta) = u_m(t, x) + \sum_{i=1}^{N} \alpha_i(t, x) \phi_i(\zeta),
\end{equation}
where $u_m$ is the depth-averaged velocity, $\phi_i$ are scaled Legendre polynomials with $\phi_i(0) = 1$, and $\alpha_i$ for $i=1, \ldots, N$ are the expansion coefficients or moments with maximum order $N$. Evolution equations for the moments $\alpha_i$ are derived by projecting the reference system onto the basis functions, resulting in the Shallow Water Moment Equations (SWME) \cite{Kowalski2019}, a closed set of $N+2$ PDEs for the variables $(h, u_m, \alpha_1, \ldots, \alpha_N)$.

Although the SWME from \cite{Kowalski2019} produce accurate results in certain academic test cases, they exhibit two key disadvantages: (1) they are not globally hyperbolic and (2) lack analytical steady states derivable from Rankine–Hugoniot conditions. The lack of hyperbolicity results in instabilities in numerical simulations \cite{Koellermeier2020c} and the lack of analytical steady states impedes the development of well-balancing numerical methods \cite{Cao2022}.

The Hyperbolic Shallow Water Moment Equations (HSWME) addressed the first issue by proving hyperbolicity via a regularization technique known as linearization around linear velocity profiles \cite{Koellermeier2020c}. The Shallow Water Linear Moment Equations (SWLME) \cite{Koellermeier2020i} resolved the second disadvantage by introducing a model that admits analytical steady states, under the assumption of small higher-order moments, though with some loss in accuracy and equilibrium stability \cite{Huang2022}.

These existing approaches mitigate individual shortcomings, but a model satisfying all desirable properties simultaneously remains unavailable. Ideally, a shallow water moment model should (1) preserve mass and momentum balance, (2) be globally hyperbolic, (3) admit analytical steady states, and (4) maintain accuracy with respect to the unregularized model.

In this paper, we present two new hierarchies of hyperbolic SWME that meet all these criteria. We first show that simply combining existing regularization strategies from \cite{Koellermeier2020c} and \cite{Koellermeier2020i} is insufficient. We then introduce a new regularization based on the equations in primitive variables and apply two different regularization techniques in this transformed setting, following ideas from hyperbolic moment models in kinetic theory \cite{Koellermeier2017d}. The resulting models are derived for arbitrary order, both in primitive and convective variables. We prove their hyperbolicity, derive analytical steady states, and evaluate their numerical accuracy (as deviation from the original SWME) using a standard dam-break test case.

The remainder of this paper is organized as follows:
Section \ref{sec:SWME} introduces the necessary notation, before deriving the analytical form of the SWME system for arbitrary order in convective and primitive variables and discussing the shortcomings of existing hyperbolic regularizations of the SWME.
Section \ref{sec:MHSWME} shows that a simple combination of existing regularization strategies, resulting in a new so-called MHSWME system, is not effective at restoring both hyperbolicity and analytical steady states.
Section \ref{sec:new_hyp} employs the new primitive variable regularization, resulting in the new so-called  PHSWME system, for which we analytically derive the equations, prove hyperbolicity and derive the steady states.
Section \ref{sec:new_new_hyp} combines the new primitive variable regularization with keeping the momentum balance equation unchanged, resulting in a new so-called PMHSWME system, for which we successfully perform the same analysis.
In section \ref{sec:numerics}, a simulation test case shows the increased accuracy of our new models and confirms that the unchanged momentum equation in the primitive regularization is crucial. The paper ends with a short conclusion and recommendations for future work.

\section{Existing Shallow Water Moment Equations (SWME)}
\label{sec:SWME}
Throughout this paper, we will derive and analyze PDE systems in the form
\begin{equation}\label{eq:p}
    \frac{\partial U_{p}}{\partial t} + A_{p} \frac{\partial U_{p}}{\partial x} = 0
\end{equation}
for the so-called primitive variables $U_{p}$ defined by
\begin{equation}\label{eq:p_vars}
    U_{p} =(h, u_m, \alpha_1, \ldots, \alpha_N) \in \mathbb{R}^{N+2}.
\end{equation}
Note that for the derivation and analysis we assume a zero right hand side since we only consider the transport part of the system \eqref{eq:p}, which is fully characterized by the system matrix $A_{p}=A_{p}(U_{p}) \in \mathbb{R}^{(N+2)\times(N+2)}$, depending nonlinearly on $U_{p}$, but we drop the dependency in the notation for brevity.

The PDE systems \eqref{eq:p} can be rewritten in the so-called convective variables $U_{c}$ defined by
\begin{equation}\label{eq:c_vars}
    U_{c} =(h, h u_m, h\alpha_1, \ldots, h\alpha_N) \in \mathbb{R}^{N+2}
\end{equation}
using the invertible transformation $U_{c} = T\left(U_{p}\right)$, i.e., $U_{p} = T^{-1}\left(U_{c}\right)$.

By simple manipulation of \eqref{eq:p}, we obtain the convective form of the equations
\begin{equation}\label{eq:c}
    \frac{\partial U_{c}}{\partial t} + A_{c} \frac{\partial U_{c}}{\partial x} = 0,
\end{equation}
using the transformed system matrix $A_{c} = \frac{\partial T\left(U_{p}\right)}{\partial U_{p}} A_{p} \left(\frac{\partial  T\left(U_{p}\right)}{\partial U_{p}}\right)^{-1}\in \mathbb{R}^{(N+2)\times(N+2)}$. Similarly, we have $A_{p} = \left(\frac{\partial T\left(U_{p}\right)}{\partial U_{p}}\right)^{-1} A_{c} \frac{\partial  T\left(U_{p}\right)}{\partial U_{p}}$.

The variable transformation uses the following transformation matrices
\begin{equation}\label{eq:dervars}
    \frac{\partial T\left(U_{p}\right)}{\partial U_{p}} =
    \left(\begin{array}{ccccc}
    1 & & & &\\
    u_m & h & & &\\
    \alpha_1 & & \ddots & &\\
    \vdots & & & \ddots &\\
    \alpha_N & & & & h
    \end{array}\right) \in \mathbb{R}^{(N+2)\times(N+2)},
\end{equation}
\begin{equation}\label{eq:inv_dervars}
    \frac{\partial T\left(U_{p}\right)}{\partial U_{p}}^{-1} =
    \left(\begin{array}{ccccc}
        1 & & & &\\
        -\frac{u_m}{h} & \frac{1}{h} & & &\\
        -\frac{\alpha_1}{h} & & \ddots & &\\
        \vdots & & & \ddots &\\
        -\frac{\alpha_w}{h} & & & & \frac{1}{h}
    \end{array}\right) \in \mathbb{R}^{(N+2)\times(N+2)}.
\end{equation}

An important property of equations like \eqref{eq:c} is hyperbolicity, for which the system matrix $A_{c}$ needs to have only real eigenvalues with a full set of eigenvectors. This is equivalent to the primitive system matrix $A_{p}$ having the same properties, since $A_{p}$ is obtained as the result of a similarity transformation. Note that while $A_{c}$ and $A_{p}$ share the same eigenvalues, the eigenvectors are different due to the transformation.
We will analyze all derived models in this paper with respect to hyperbolicity and see that some of them are hyperbolic regardless of the values of $U_{c}$, making them globally hyperbolic. Others are only hyperbolic in a certain range of values $U_{c}$, while loosing hyperbolicity outside that range, making them only locally hyperbolic.

\subsection{SWME in convective and primitive variables}
\label{sec:SWME_c_p}
We review the SWME derived in \cite{Kowalski2019} and write them for the first time in both convective variables \eqref{eq:c} and primitive variables \eqref{eq:p}, including the analytical convective system matrix $A_{c}^{\textrm{SWME}}$ and primitive system matrix $A_{p}^{\textrm{SWME}}$.
This serves as the basis of obtaining new hyperbolic moment models in the next sections.

The frictionless 1D SWME system as derived in \cite{Kowalski2019}, consists of $N+2$ equations for $h$, $hu_m$ and $h \alpha_i$ for $i=1, \ldots, N$ and is written as
\begin{eqnarray}
    \partial_t h+\partial_x\left(h u_m\right) &=& 0, \label{eq:SWME_h}\\
    \partial_t\!\left(h u_m\right)+\partial_x\!\left(\!h\!\left(\!u_m^2+\sum_{j=1}^N \frac{\alpha_j^2}{2 j+1}\!\right)\!+\!\frac{1}{2} g h^2\!\right)\! &=& 0, \label{eq:SWME_um}\\
    \partial_t\!\left(h \alpha_i\right)+\partial_x\!\left(\!h\!\left(\!2 u_m \alpha_i+\sum_{j, k=1}^N A_{i j k} \alpha_j \alpha_k\!\right)\!\right)\! &=&  u_m \partial_x\!\left(h \alpha_i\right)\! -\! \sum_{j, k=1}^N \!B_{i j k} \alpha_k \partial_x\left(h \alpha_j\right)\!, \label{eq:SWME_alphai}
\end{eqnarray}
where the following coefficients are used
\begin{equation}\label{eq:Aijk_Bijk}
    A_{i j k}=(2 i+1) \int_0^1 \phi_i \phi_j \phi_k d \zeta, \quad B_{i j k}=(2 i+1) \int_0^1 \phi_i^{\prime}\left(\int_0^\zeta \phi_j d \zeta\right) \phi_k d \zeta.
\end{equation}
Note that we consider the 1D case only for simplicity. As was recently shown, the extension of existing models to either 2D axisymmetric settings \cite{Verbiest2025} or full 2D settings \cite{Bauerle2025} is possible and not considered in this work. The frictionless case is considered as we focus on the transport properties.

The following lemma states the form of the convective system matrix.
\begin{lemma}
    The SWME \eqref{eq:SWME_h}-\eqref{eq:SWME_alphai} can be written in the form of \eqref{eq:c} using the convective variables $U_{c}$ \eqref{eq:c_vars} as
    \begin{equation}\label{eq:SWME_c}
        \frac{\partial U_{c}}{\partial t} + A^{\textrm{SWME}}_{c} \frac{\partial U_{c}}{\partial x} = 0,
    \end{equation}
    using the convective SWME system matrix $A^{\textrm{SWME}}_{c} \in \mathbb{R}^{(N+2)\times(N+2)}$ defined as
    \begin{equation}\label{eq:A_SWME_c}
        A^{\textrm{SWME}}_{c} =
        \left(\begin{array}{ccccc}
        0 & 1 & 0 & \cdots & 0 \\
        - u_m^2 + gh - \sum\limits_{j=1}^{N} \frac{\alpha_j^2}{2j+1} & 2 u_m & \frac{2}{3} \alpha_1 & \cdots & \frac{2}{2 N+1} \alpha_N \\
        -2 u_m \alpha_1 - \sum\limits_{j,k=1}^{N} A_{1jk} \alpha_j \alpha_k & 2 \alpha_1 &  & & \\
        \vdots & \vdots & & \mathcal{A} &  \\
        -2 u_m \alpha_N - \sum\limits_{j,k=1}^{N} A_{Njk} \alpha_j \alpha_k & 2 \alpha_N &  & &
        \end{array}\right),
    \end{equation}
    where the lower right block matrix $\mathcal{A} \in \mathbb{R}^{N \times N}$ is defined by
    \begin{equation}\label{eq:A_block}
        \mathcal{A}_{i,l} = \sum\limits_{j=1}^{N} \left( B_{ilj} + 2 A_{ijl} \right) \alpha_j + u_m \delta_{i,l},
    \end{equation}
    with Kronecker delta $\delta_{i,j}$ and $A,B$, defined in \eqref{eq:Aijk_Bijk}.
\end{lemma}
\begin{proof}
    We give a constructive proof and derive each row of the matrix $A^{\textrm{SWME}}_{c}$.

    \noindent(1): The first row is trivially given due to the first equation from \eqref{eq:SWME_h}.

    \noindent(2): For the second row, we note for the individual terms
        \begin{eqnarray}
            \partial_x\left(h u_m^2\right)& =& -u_m^2 \partial_x h+2 u_m \partial_x\left(h u_m\right), \\
            \partial_x\left(\frac{g}{2} h^2\right)& =&g h \partial_x h, \\
            \partial_x\left(\sum_{j=1}^N \frac{h \alpha_j^2}{2 j+1}\right)& =&\ldots=-\sum_{j=1}^N \frac{1}{2 j+1} \alpha_j^2 \partial_x h+\sum_{j=1}^N \frac{2}{2 j+1} \alpha_j \partial_x\left(h \alpha_j\right).
        \end{eqnarray}
        So that we can write the second equation from \eqref{eq:SWME_um} as
        \begin{equation}
            \partial_{x}\left(h u_m\right)+\left(g h-u_m^2-\sum_{j=1}^N \frac{\alpha_j^2}{2j+1}\right) \partial_x h+2 u_m \partial_x\left(h u_m\right)+\sum_{j=1}^N \frac{2}{2 j+1} \alpha_j \partial_x\left(h \alpha_j\right)=0.
        \end{equation}

    \noindent$(2+i)$: For the remaining moment equations for $i=1,\ldots, N$, we first note that
        \begin{eqnarray}
            \partial_x\left(2 u_m h \alpha_i\right)-u_m \partial_x\left(h \alpha_i\right) & =&2 u_m \partial_x\left(h \alpha_i\right)+2 h \alpha_i \partial_x\left(u_m\right)-u_m \partial_x\left(h \alpha_i\right) \\
            & =&u_m \partial_x\left(h \alpha_i\right)-2 u_m \alpha_i \partial_x h+2 u_m \alpha_i \partial_x h+2 h \alpha_i \partial_x u_m \\
            & =&u_m \partial_x\left(h \alpha_i\right)-2 u_m \alpha_i \partial_x h+2 \alpha_i \partial_x\left(h u_m\right).
        \end{eqnarray}
        In the same way, we can derive
        \begin{eqnarray}
            \partial_x \sum_{j,k=1}^N A_{i j k} h \alpha_j \alpha_k & =&\sum_{j,k=1}^N A_{i j k} \left( \alpha_j \alpha_k \partial_x h+ \alpha_{j} h \partial_x \alpha_k +  \alpha_k h \partial_x \alpha_j \right) \\
            & = &\sum_{j,k=1}^N A_{i j k} \alpha_j \partial_x\left(h \alpha_k\right)+A_{i j k} \alpha_k \partial_x\left(h \alpha_j\right)-A_{i j k} \alpha_j \alpha_k \partial_x h \\
            & =& \sum_{j,k=1}^N 2 A_{i j k} \alpha_j \partial_x\left(h \alpha_k\right)-\sum_{j,k=1}^N A_{i j k} \alpha_j \alpha_k \partial_x h,
        \end{eqnarray}
        where the last step used the symmetry of $A_{i j k}=A_{i k j}$.
        This means that the $i-$th equation of \eqref{eq:SWME_alphai} can be written as
        \begin{eqnarray}
            \partial_t\left(h \alpha_i\right) + \left( -2u_m \alpha_1 -\sum_{j,k=1}^N A_{i j k} \alpha_j \alpha_k \right) \partial_x h + 2\alpha_i \partial_x \left(h u_m \right) + \sum_{l=1}^N \left( \mathcal{A}_{i,l} \right) \partial_x \left( h \alpha_l\right),
        \end{eqnarray}
        for $\mathcal{A}_{i,l} = \sum\limits_{j=1}^N \left(B_{ilj} + 2 A_{i j l}\right) \alpha_j + u_m \delta_{i,l}$ as defined in \eqref{eq:A_block}.

        The elements of the matrix $A^{\textrm{SWME}}_{c}$ can then directly be read as the coefficients in front of the derivatives.
\end{proof}

The next lemma states the system matrix in the primitive variables.
\begin{lemma}\label{lemma:SWME_p}
    The SWME \eqref{eq:SWME_h}-\eqref{eq:SWME_alphai} and its convective form \eqref{eq:A_SWME_c} can be written in the primitive form of \eqref{eq:p} using the primitive variables $U_{p}$ \eqref{eq:p_vars} as
    \begin{equation}\label{eq:SWME_p}
        \frac{\partial U_{p}}{\partial t} + A^{\textrm{SWME}}_{p} \frac{\partial U_{p}}{\partial x} = 0,
    \end{equation}
    using the primitive SWME system matrix $A^{\textrm{SWME}}_{p} \in \mathbb{R}^{(N+2)\times(N+2)}$ defined as
    \begin{equation}\label{eq:A_SWME_p}
        A^{\textrm{SWME}}_{p} =
        \left(\begin{array}{ccccc}
        u_m & h & 0 & \cdots & 0 \\
        g + \frac{1}{h}\sum\limits_{j=1}^{N} \frac{\alpha_j^2}{2j+1} & u_m & \frac{2}{3} \alpha_1 & \cdots & \frac{2}{2 N+1} \alpha_N \\
        \frac{1}{h}\sum\limits_{j,k=1}^{N} \left(B_{1jk} + A_{1jk} \right) \alpha_j \alpha_k & \alpha_1 &  & & \\
        \vdots & \vdots & & \mathcal{A} &  \\
        \frac{1}{h}\sum\limits_{j,k=1}^{N} \left(B_{Njk} + A_{Njk} \right) \alpha_j \alpha_k & \alpha_N &  & &
        \end{array}\right),
    \end{equation}
    with the lower right block matrix $\mathcal{A} \in \mathbb{R}^{N \times N}$ from \eqref{eq:A_block} and $A,B$, defined in \eqref{eq:Aijk_Bijk}.
\end{lemma}
\begin{proof}
    While it is possible to derive the primitive system directly from by computing the Jacobian in primitive variables from equations \eqref{eq:SWME_h}, \eqref{eq:SWME_um}, \eqref{eq:SWME_alphai}, we here use the variable transformation defined in the beginning of section \ref{sec:SWME} to transform between both sets of variables. This will later be useful for other models.

    We start from the convective system matrix $A^{\textrm{SWME}}_{c}$ and compute the primitive system matrix using the variable transformation $A^{\textrm{SWME}}_{p} = \frac{\partial T\left(U_{p}\right)}{\partial U_{p}}^{-1} A^{\textrm{SWME}}_{c} \frac{\partial  T\left(U_{p}\right)}{\partial U_{p}}$ with the transformation matrices given in \eqref{eq:dervars} and \eqref{eq:inv_dervars}.

    The first step is the multiplication $\frac{\partial T\left(U_{p}\right)}{\partial U_{p}}^{-1} A^{\textrm{SWME}}_{c}$. The effect of multiplication with $\frac{\partial T\left(U_{p}\right)}{\partial U_{p}}^{-1}$ from the left can be seen as division of most entries by $h$ and addition of a multiple of the first row to the second row. Since only the second element in the first row of $A^{\textrm{SWME}}_{c}$ is non-zero, this effectively only changes the second column. Using the convective system matrix from \eqref{eq:A_SWME_c}, we derive
    \begin{eqnarray}
        \frac{\partial T\left(U_{p}\right)}{\partial U_{p}}^{-1} A^{\textrm{SWME}}_{c} 
        &=& \frac{1}{h} \left(\begin{array}{ccccc}
        0 & h & 0 & \cdots & 0 \\
        - u_m^2 + gh - \sum\limits_{j=1}^{N} \frac{\alpha_j^2}{2j+1} & u_m & \frac{2}{3} \alpha_1 & \cdots & \frac{2}{2 N+1} \alpha_N \\
        -2 u_m \alpha_1 - \sum\limits_{j,k=1}^{N} A_{1jk} \alpha_j \alpha_k & \alpha_1 &  & & \\
        \vdots & \vdots & & \mathcal{A} &  \\
        -2 u_m \alpha_N - \sum\limits_{j,k=1}^{N} A_{Njk} \alpha_j \alpha_k & \alpha_N &  & &
        \end{array}\right)
    \end{eqnarray}

    The last step is the multiplication $\left( \frac{\partial T\left(U_{p}\right)}{\partial U_{p}}^{-1} A^{\textrm{SWME}}_{c}\right) \frac{\partial  T\left(U_{p}\right)}{\partial U_{p}}$.
    We first note that
    \begin{equation}
        \frac{1}{h} \cdot \frac{\partial  T\left(U_{p}\right)}{\partial U_{p}} = \left(\begin{array}{ccccc}
        \frac{1}{h} & & & &\\
         \frac{u_m}{h}& 1 & & &\\
         \frac{\alpha_1}{h}& & \ddots & &\\
        \vdots & & & \ddots &\\
        \frac{\alpha_N }{h}& & & & 1
        \end{array}\right)
    \end{equation}
    The effect of multiplication with $\frac{1}{h}\frac{\partial T\left(U_{p}\right)}{\partial U_{p}}$ from the right can thus be seen as keeping the columns $2,\ldots, N+2$ constant, but replacing the entries in the first column by a weighted sum of the whole column. Using the result of the first step, the new first column is computed as:
    \begin{itemize}
        \item[(1)] $(A^{\textrm{SWME}}_p)_{1,1} = h \frac{u_m}{h} = u_m$,
        \item[(2)] $(A^{\textrm{SWME}}_p)_{2,1} = \left( - u_m^2 + gh - \sum\limits_{j=1}^{N} \frac{\alpha_j^2}{2j+1}\right) \frac{1}{h} + u_m \frac{u_m}{h} + \sum\limits_{j=1}^{N} \frac{2}{2 j+1} \alpha_j \frac{\alpha_j}{h} = g + \frac{1}{h} \sum\limits_{j=1}^{N} \frac{\alpha_j^2}{2 j+1} $,
        \item[(2+i)] $(A^{\textrm{SWME}}_p)_{2+i,1} = \left( -2 u_m \alpha_i - \sum\limits_{j,k=1}^{N} A_{ijk} \alpha_j \alpha_k \right) \frac{1}{h} + \alpha_i \frac{u_m}{h} + \sum\limits_{k=1}^{N} \mathcal{A}_{k,j} \frac{\alpha_k}{h}$.

        Noting the definition of $\mathcal{A}_{i,k} = \sum\limits_{j=1}^{N} \left( B_{ikj} + 2 A_{ijk} \right) \alpha_j + u_m \delta_{i,k}$, we get
        \begin{eqnarray}
            (A^{\textrm{SWME}}_p)_{2+i,1} &=& \frac{1}{h}\sum\limits_{j,k=1}^{N} \left(B_{ijk} + A_{ijk} \right) \alpha_j \alpha_k.
        \end{eqnarray}
    \end{itemize}
    This completes the proof.
\end{proof}

\subsection{Existing hyperbolic regularizations of SWME}
Hyperbolicity of the SWME \eqref{eq:SWME_c} was extensively studied in \cite{Koellermeier2020c}. While the SWME is globally hyperbolic for $N=1$, it is only locally hyperbolic for $N\geq 2$. Importantly, for larger number of equations $N$ the SWME already looses hyperbolicity for infinitely small deviations from equilibrium, represented by small coefficients $\alpha_i$.

Several hyperbolic regularizations have been developed in \cite{Koellermeier2020c} and \cite{Koellermeier2020i} aiming at overcoming the loss of hyperbolicity.
These existing hyperbolic regularizations of the SWME are all based on the convective variables representation \eqref{eq:SWME_c}. We will briefly state the main ideas and the resulting system matrices for two popular hyperbolic regularizations, the HSWME derived in \cite{Koellermeier2020c} and the SWLME derived in \cite{Koellermeier2020i}, respectively.

We note that comparable hyperbolic regularizations were also derived in the field of rarefied gas dynamics, see e.g., \cite{Cai2012a,Cai2013b,Cai2014a,Fan2016}, and the related works \cite{Schaerer2015,McDonald2013}.

\subsubsection{Hyperbolic Shallow Water Moment Equations (HSWME)}
The so-called Hyperbolic Shallow Water Moment Equations (HSWME), were derived in \cite{Koellermeier2020c} based on the observation that the SWME model from \eqref{eq:SWME_c} is still hyperbolic for $N=1$. Keeping the first coefficient $\alpha_1$, but setting all higher-order expansion coefficients $\alpha_i=0$ for $i \geq 2$, in the system matrix $A^{\textrm{SWME}}_{c}$ from \eqref{eq:SWME_c}, i.e. $A^{\textrm{SWME}}_{c}\left(h,hu_m,h \alpha_1, 0, \ldots, 0\right) = A^{\textrm{HSWME}}_{c}$ leads to the new system matrix as shown in \cite{Koellermeier2020c}, resulting in
\begin{equation}\label{eq:A_HSWME_c}
    A^{\textrm{HSWME}}_{c} =
    \left(\begin{array}{cccccc}
         & 1 &  &  &  &  \\
        -u_m^2 + gh - \frac{\alpha_1^2}{3} & 2u_m & \frac{2}{3} \alpha_1 & &\\
        -2u_m \alpha_1 & 2\alpha_1 & u_m & \frac{3}{5} \alpha_1 & &\\
        -\frac{2}{3}\alpha_1^2 &   & \frac{1}{3}\alpha_1 & u_m & \ddots & \\
         &  &  & \ddots & \ddots & \frac{N+1}{2N+1}\alpha_1 \\
         &  &  &  & \frac{N-1}{2N-1}\alpha_1 & u_m
    \end{array}\right),
\end{equation}
where all omitted entries are zero.

For the HSWME system matrix \eqref{eq:A_HSWME_c}, it is possible to compute analytically the characteristic polynomial, based on successive development with respect to rows and columns. For details we refer to the proof of Theorem 3.2 in \cite{Koellermeier2020c} and we here only state the resulting characteristic polynomial $\chi_c^{\textrm{HSWME}}$ as
\begin{equation}\label{eq:chi_HSWME_c}
    \chi_c^{\textrm{HSWME}}(\lambda) = \chi_{A_2}(\lambda- u_m) \cdot \left( (\lambda-u_m)^2 -gh-\alpha_1^2\right),
\end{equation}
where $\chi_{A_2}$ is the characteristic polynomial of the matrix $A_2$ defined by
\begin{equation}\label{eq:A_2}
        A_2 =
        \left(\begin{array}{cccc}
         & c_2 & &  \\
        a_2 & & \ddots &  \\
         & \ddots & & c_N \\
         &  & a_N &  \\
         \end{array}\right) \in \mathbb{R}^{N \times N},
\end{equation}
with off-diagonal elements $a_i= \frac{i-1}{2i-1} \alpha_1$ and $c_i= \frac{i+1}{2i+1} \alpha_1$, for $i=2, \ldots, N$.

The matrix $A_2$ \eqref{eq:A_2} is the lower-right block part of the convective system matrix $A^{\textrm{HSWME}}_{c}$ \eqref{eq:A_HSWME_c} without diagonal. Its entries can thus be interpreted as the remaining coupling terms between the moment coefficients in the higher order equations of the model. Note that the full (non hyperbolic) SWME system matrix $A^{\textrm{SWME}}_{c}$ \eqref{eq:SWME_c} includes a more dense counterpart $\mathcal{A}$ given by \eqref{eq:A_block} instead. The simplification from $\mathcal{A}$ to $A_2$ is due to the linearization of the system matrix around linear velocity profiles, significantly simplifying the computation of eigenvalues.

In \cite{Koellermeier2020c}, the roots of $\chi_{A_2}(\lambda)$ were only shown to be real by numerically computing them for some $N$. However, as a variant of a similar result proven in \cite{Bauerle2025}, we can prove the following statement.

\begin{lemma}\label{lemma:A_2_chi}
    The matrix $A_2 \in \mathbb{R}^{N \times N}$ in \eqref{eq:A_2} has the following characteristic polynomial
    \begin{equation}\label{eq:A_2_char_poly}
        \chi_{A_2}(\lambda) = \frac{(-\alpha_1)^NN!}{(2N+1)!!} \cdot P^{'}_{N+1}\left(\frac{\lambda}{\alpha_1}\right),
    \end{equation}
    where $P'_{N+1}$ is the derivative of the $N+1$-degree Legendre polynomial, orthogonal on $[-1,1]$ and normalized with $P_{N}(1)=1$.
\end{lemma}
\begin{proof}
    We denote that $A_2 \in \mathbb{R}^{N\times N}$ in \eqref{eq:A_2} is a tridiagonal matrix, a special form of unreduced Hessenberg matrix. We can therefore proceed analogously as in \cite{Bauerle2025}, where the matrix $A_2$ is part of a slightly larger matrix called $\tilde{A}_{11}$. We note that the characteristic polynomial of $A_2$ can be computed in terms of a so-called associated polynomial sequence $q_i$ for $i=0,\ldots, N$. For the tridiagonal matrix $A_2$ in \eqref{eq:A_2}, the associated polynomial sequence $q_i$ is defined as
    \begin{eqnarray}
        q_0(\lambda) &=& 1, \label{eq:q_sequence1}\\
        q_i(\lambda) &=& \frac{1}{c_{i+1}} \left( \lambda q_{i-1}(\lambda) - a_i q_{i-2}(\lambda) \right), \textrm{ for } i=1,\ldots, N. \label{eq:q_sequence2}
    \end{eqnarray}
    Similar to \cite{Bauerle2025}, we can prove that the closed form of the sequence $q_i$ is given by
    \begin{equation}
        q_i(\lambda) = \frac{2(2i+3)}{3(i+2)(i+1)\alpha_1} P^{'}_{i+1}\left(\frac{\lambda}{\alpha_1}\right), \textrm{ for } i=1,\ldots, N-1,
    \end{equation}
    where $P^{'}_{N+1}$ is the derivative of the $N+1$-degree Legendre polynomial, orthogonal on $[-1,1]$ and normalized with $P_{N}(1)=1$. This follows directly from induction using
    \begin{eqnarray}
        q_{i+1}(\lambda) &=& \frac{1}{c_{i+2}} \left( \lambda q_{i}(\lambda) - a_{i+1} q_{i-1}(\lambda) \right) \\
        &=& \frac{2i+5}{(i+3) \alpha_1} \!\left( \!\lambda \frac{2(2i+3)}{3(i+2)(i+1)\alpha_1} P^{'}_{i+1}\!\left(\frac{\lambda}{\alpha_1}\right) \!-\! \frac{i \alpha_1}{2i+1} \frac{2(2i+1)}{3(i+1) i \alpha_1} P^{'}_{i}\!\left(\frac{\lambda}{\alpha_1}\right) \!\right) \\
        &=& \frac{2(2i+5)}{3(i+3)(i+2) \alpha_1} \!\left(\! \frac{\lambda}{\alpha_1} \frac{2i+3}{i+1} P^{'}_{i+1}\!\left(\frac{\lambda}{\alpha_1}\right) \!- \!\frac{i+2}{i+1} P^{'}_{i}\left(\frac{\lambda}{\alpha_1}\right) \right) \\
        &=& \frac{2(2i+5)}{3(i+3)(i+2)\alpha_1} P^{'}_{i+2}\!\left(\frac{\lambda}{\alpha_1}\right),
    \end{eqnarray}
    where we used the Legendre recursion formula $(i-1) P^{'}_i(\lambda) = (2i-1) \lambda P^{'}_{i-1} (\lambda)- i P^{'}_{i-2}(\lambda)$.

    We then directly derive that
    \begin{eqnarray}
        q_{N}(\lambda) &=& \frac{1}{c_{N+1}} \left( \lambda q_{N-1}(\lambda) - a_{N} q_{N-2}(\lambda) \right), \quad c_{N+1} =1, \quad a_N= \frac{N-1}{2N-1} \\
        &=& \lambda \frac{2(2N+1)}{3(N+1)N\alpha_1} P^{'}_{N}\left(\frac{\lambda}{\alpha_1}\right) - \frac{(N-1) \alpha_1}{2N-1} \frac{2(2N-1)}{3N (N-1) \alpha_1} P^{'}_{N-1}\left(\frac{\lambda}{\alpha_1}\right)  \\
        &=& \frac{2}{3} \frac{1}{N+1} \left( \frac{\lambda}{\alpha_1} \frac{2N+1}{N} P^{'}_{N}\left(\frac{\lambda}{\alpha_1}\right) - \frac{N+1 }{N} P^{'}_{N-1}\left(\frac{\lambda}{\alpha_1}\right) \right) \\
        &=& \frac{2}{3} \frac{1}{N+1} P^{'}_{N+1}\left(\frac{\lambda}{\alpha_1}\right),
    \end{eqnarray}
    again using the recurrence formula for $P^{'}_i$ in the last step.

    The characteristic polynomial $\chi_{A_2}(\lambda)$ is then computed using the identity matrix $I_N\in \mathbb{R}^{N \times N}$ according to
    \begin{eqnarray}
        \chi_{A_2}(\lambda) &=& \det(A_2- \lambda I_N) = (-1)^N \prod\limits_{i=2}^{N} c_i \cdot q_N(\lambda) = \frac{2(-1)^N \alpha_1}{3(N+1)} \prod\limits_{i=2}^{N} c_i \cdot P^{'}_{N+1}\left(\frac{\lambda}{\alpha_1}\right)\\
        &=& \frac{(-\alpha_1)^NN!}{(2N+1)!!} \cdot P^{'}_{N+1}\left(\frac{\lambda}{\alpha_1}\right),
    \end{eqnarray}
    similar to \cite{Bauerle2025} Theorem 4.3.
    This completes the proof.
\end{proof}

The result of Lemma \ref{lemma:A_2_chi} means that the eigenvalues of the HSWME system matrix are given by
\begin{eqnarray}\label{eq:HSWME_c_EV}
    \lambda_i, i=1,\ldots, N: P'_{N+1}\left(\frac{\lambda_i - u_m}{\alpha_1}\right)=0, \quad
    \lambda_{N+1,N+2} = u_m \pm \sqrt{gh + \alpha_1^2},
\end{eqnarray}
which are indeed real, resulting in a globally hyperbolic model.

We note that the hyperbolic regularization neglecting all higher-order coefficients in the system matrix is a severe simplification leading to potentially lower accuracy of the model. More precisely, model discrepancies can occur for non negligible values of $\alpha_i$ for $i\geq2$.


\begin{remark}
    A variant of the HSWME called $\beta$-HSWME uses a slightly different last row of the system matrix to obtain different eigenvalues. The interested reader is referred to \cite{Koellermeier2020c}, where this is discussed in more detail.
\end{remark}

\subsubsection{Shallow Water Linearized Moment Equations (SWLME)}
In \cite{Koellermeier2020i}, steady states of shallow water moment models were analyzed and it became clear that neither the original SWME nor the HSWME allow for an analytical computation of steady states for $N\geq 2$.

The paper then derives a simplified system that succeeds in having both global hyperbolicity and analytically computable steady states, called the Shallow Water Linearized Moment Equations (SWLME). The idea is that the higher-order moment equations are linearized around small values of the coefficients $\alpha_i=\mathcal{O}(\epsilon)$, neglecting terms of higher-order in $\epsilon$. Notably, this linearization is only applied in the last $N$ moment equations, while the mass and the momentum balance equations remain unchanged. This results in the following system matrix as derived in \cite{Koellermeier2020i}
\begin{equation}\label{eq:A_SWLME_c}
        A^{\textrm{SWLME}}_{c}=
        \left(\begin{array}{cccccc}
         & 1 &  &  &  &  \\
        -u_m^2 + gh - \sum\limits_{i=1}^{N}\frac{3 \alpha_1^2}{2i+1} & 2u_m & \frac{2}{3} \alpha_1 & \ldots & \ldots & \frac{2}{2N+1} \alpha_N\\
        -2u_m \alpha_1 & 2\alpha_1 & u_m & & &\\
        -2u_m \alpha_2 & 2\alpha_2  &  & u_m &  & \\
        \vdots & \vdots &  &  & \ddots &  \\
        -2u_m \alpha_N & 2\alpha_N &  &  &  & u_m
        \end{array}\right).
\end{equation}
It was shown that the eigenvalues of $A^{\textrm{SWLME}}_{c}$ are always real and steady states are given as a generalization of the standard SWE. We here only show the eigenvalues as a result from \cite{Koellermeier2020i}
\begin{eqnarray}\label{eq:SWLME_c_EV}
    \lambda_i=u_m, \, i=1,\ldots, N, \quad
    \lambda_{N+1,N+2} = u_m \pm \sqrt{gh + \alpha_1^2 + \sum\limits_{i=2}^N\frac{\alpha_i^2}{2N+1}},
\end{eqnarray}
which are notably simpler than the ones for the HSWME \eqref{eq:HSWME_c_EV}.

In \cite{Koellermeier2020i}, steady states of the SWLME system were derived as generalizations of steady states for the standard SWE and read
\begin{equation}\label{eq:SWLME_steady_states}
    h=h_0 ~\textrm{ or } ~ -\!\Fr^2 + \frac{1}{2} \left(\!\left(\frac{h}{h_0}\right)^2\!+\!\left(\frac{h}{h_0}\right)\!\right) + \sum\limits_{i=1}^{N}\frac{\Ma_i^2 \Fr^2}{2i+1}  \left(\! \left(\frac{h}{h_0}\right)^3\!+\!\left(\frac{h}{h_0}\right)^2\!+\!\left(\frac{h}{h_0}\right)\!\right)\!=\!0,
\end{equation}
where, in addition to the Froude number $\Fr=\frac{u_m}{\sqrt{gh}}$, $\Ma_i = \frac{\alpha_{i,0}}{u_{m,0}}$ denotes the dimensionless $i$-th moment number for $i=1, \ldots, N$ at a reference state.
Note how for $\alpha_i = 0$ the steady states revert to the standard SWE steady states for a hydraulic water jump.

While the SWLME \eqref{eq:A_SWLME_c} are successful in obtaining hyperbolicity and analytical steady states, they lack physical accuracy as already demonstrated in \cite{Huang2022}, due to the lack of non-linear moment equations modeling the evolution of the expansion coefficients. It seems that the regularization employed by the SWLME is too strong an assumption to preserve accuracy.

\section{Moment regularization for Hyperbolic Shallow Water Moment Equations (MHSWME)}
\label{sec:MHSWME}
Based on the results of the previous two sections, we can conclude that two approaches successfully lead to a regularized, globally hyperbolic system of moment equations:
\begin{enumerate}
    \item[(1)] HSWME approach: in the convective system matrix, setting all higher coefficients $\alpha_i=0$ for $i\geq 2$, and
    \item[(2)] SWLME approach: in the convective system matrix, keeping the momentum balance equation exactly and assuming only in the higher moment equations small coefficients $\alpha_i=\mathcal{O}(\epsilon)$ for $i\geq 1$, so that contributions of $\mathcal{O}(\epsilon^2)$ can be neglected.
\end{enumerate}
Both approaches start with the convective system matrix. A main difference between approaches (1) and (2) is the treatment of the momentum balance equation.

An obvious question is now what happens if the HSWME approach (1) is performed while keeping the momentum equation unchanged.

In this section, we will perform the regularization in convective variables by setting only in the higher moment equations all higher coefficients $\alpha_i=0$ for $i\geq 2$. We consider this as a Moment regularization for Hyperbolic Shallow Water Moment Equations and denote the resulting system as MHSWME. As we will see, the MHSWME system is locally hyperbolic in a reasonably large and non-decreasing domain around equilibrium, but not globally hyperbolic. Also, steady states are not obtainable analytically.

\begin{theorem}\label{th:hybrid_convective_system}
    Setting the higher-order coefficients $\alpha_i=0$ for $i\geq 2$ in the convective system matrix \eqref{eq:A_SWME_c} only in the last $N$ moment equations, results in the following MHSWME system matrix in convective variables:
    \begin{equation}\label{eq:A_MHSWME_c}
        A^{\textrm{MHSWME}}_{c} =
        \left(\begin{array}{cccccc}
             & 1 &  &  &  &  \\
        - u_m + gh - \sum\limits_{j=1}^{N} \frac{\alpha_j^2}{2j+1} & 2 u_m & \frac{2}{3} \alpha_1 & \frac{2}{5} \alpha_2 &\cdots & \frac{2}{2 N+1} \alpha_N \\
            -2u_m \alpha_1 & 2\alpha_1 & u_m & \frac{3}{5} \alpha_1 & &\\
            -\frac{2}{3}\alpha_1^2 &   & \frac{1}{3}\alpha_1 & u_m & \ddots & \\
             &  &  & \ddots & \ddots & \frac{N+1}{2N+1}\alpha_1 \\
             &  &  &  & \frac{N-1}{2N-1}\alpha_1 & u_m
        \end{array}\right),
    \end{equation}

    which has the characteristic polynomial
    \begin{equation}\label{eq:A_MHSWME_char}
        \chi_{A^{\textrm{MHSWME}}_{c}}(\lambda) = \frac{(-\alpha_1)^NN!}{(2N+1)!!} \cdot P^{'}_{N+1}\left(\frac{\lambda-u_m}{\alpha_1}\right) \cdot \left( \left( \lambda - u_m \right)^2 - gh - \alpha_1^2 + \sum\limits_{i=2}^N\frac{\alpha_i^2}{2i+1}\right),
    \end{equation}
    and the eigenvalues $\lambda_1, \ldots, \lambda_{N+2}$ of $A^{\textrm{MHSWME}}_{c}$ are given by
    \begin{eqnarray}\label{eq:A_MHSWME_EV}
        \lambda_i, i=1,\ldots, N: P'_{N+1}\left(\frac{\lambda_i - u_m}{\alpha_1}\right)=0,
        \lambda_{N+1,N+2} = u_m \pm \sqrt{gh + \alpha_1^2 - \sum\limits_{i=2}^N\frac{\alpha_i^2}{2i+1}}.\\
    \end{eqnarray}
\end{theorem}
\begin{proof}
    The first two rows of the convective system matrix are taken from the general convective system matrix given in \eqref{eq:SWME_c} without changes. The higher-order moment equations, i.e., the last $N$ equations, are the same as in \eqref{eq:A_HSWME_c}.

    For the derivation of the characteristic polynomial, we follow the proof of Theorem 3.2 in \cite{Koellermeier2020c}, which also used the convective form of the equations.
    Starting with
    \begin{equation}
       A^{\textrm{MHSWME}}_{c} = \widetilde{A}^{\textrm{MHSWME}} + u_m I_{N+2}, \quad \widetilde{\lambda} = \lambda - u_m,
    \end{equation}
    we will compute
    \begin{equation}
       \chi_{A^{\textrm{MHSWME}}_{c}}(\lambda) = \det\left( \widetilde{A}^{\textrm{MHSWME}} - \widetilde{\lambda} I_{N+2} \right) =: \left| \widetilde{A}^{\textrm{MHSWME}} - \widetilde{\lambda} I_{N+2} \right|,
    \end{equation}
    with
    \begin{equation}
        \widetilde{A}^{\textrm{MHSWME}} - \widetilde{\lambda} I_{N+2} = \left(\begin{array}{cccccc}
        - \widetilde{\lambda}-u_m & 1 &  &  &  &  \\
        d_1 + \widetilde{d}_{1} & - \widetilde{\lambda} + u_m & d_2 & \widetilde{d}_{3} & \dots & \widetilde{d}_{{N+1}} \\
        d_3 & d_4 & - \widetilde{\lambda} & c_2 & &\\
        d_5 &   & a_2 & - \widetilde{\lambda} & \ddots & \\
         &  &  & \ddots & \ddots & c_N \\
         &  &  &  & a_N & - \widetilde{\lambda}
        \end{array}\right),
    \end{equation}
    where we use the following definitions
    \begin{equation}
        d_1 = gh - u_m^2 -\frac{1}{3}\alpha_1^2, \quad d_2 = \frac{2}{3}\alpha_1, \quad d_3 = -2 u_m \alpha_1, \quad d_4 = 2\alpha_1, \quad d_5 = -\frac{2}{3}\alpha_1^2,
    \end{equation}
    \begin{equation}
       \widetilde{d}_{1} = - \sum\limits_{i=2}^N \frac{\alpha_i^2}{2i+1}, \quad \widetilde{d}_{i} = \frac{2}{2i+1}\alpha_{i-1}, \textrm{ for } i=2,\ldots, N+1,
    \end{equation}
    and as usual we denote $a_i = \frac{i-1}{2i-1}\alpha_1$ and $c_i = \frac{i+1}{2i+1}\alpha_1$, for $i=2,\ldots, N$.
    Note how the entries $\widetilde{d}_{i}$ denote entries that are different with respect to the standard convective HSWME matrix in \eqref{eq:A_HSWME_c} and in the proof of Theorem 3.2 in \cite{Koellermeier2020c}.

    We continue to compute $\left| \widetilde{A}^{\textrm{MHSWME}} - \widetilde{\lambda} I_{N+2} \right|$ by first developing with respect to the first row and then with respect to the first column.
    \begin{eqnarray*}
    && \left| \widetilde{A}^{\textrm{MHSWME}} - \widetilde{\lambda} I_{N+2} \right| = \left|
        \begin{array}{cccccc}
            -\widetilde{\lambda} - u_m  & 1 &  &  &  & \\
            d_1+\widetilde{d}_{1} & -\widetilde{\lambda} + u_m & d_2 & \widetilde{d}_{3} & \dots & \widetilde{d}_{{N+1}}\\
            d_3 & d_4 & -\widetilde{\lambda} & c_2  &  &   \\
            d_5 &  & a_2 & -\widetilde{\lambda} & \ddots  &   \\
            &  &  & \ddots & \ddots & c_N \\
            &  &  &  & a_N & -\widetilde{\lambda}  \\
        \end{array}
        \right| \\
      &=& (-\widetilde{\lambda} - u_m)
      \left|
        \begin{array}{ccccc}
            -\widetilde{\lambda} + u_m & d_2 & \widetilde{d}_{3} & \dots & \widetilde{d}_{{N+1}}\\
            d_4 & -\widetilde{\lambda} & c_2  &  &   \\
            & a_2 & -\widetilde{\lambda} & \ddots  &   \\
            &  & \ddots & \ddots & c_N \\
            &  &  & a_N & -\widetilde{\lambda}   \\
        \end{array}
        \right|
        -1 \left|
        \begin{array}{ccccc}
            d_1+\widetilde{d}_{{1}} & d_2 & \widetilde{d}_{3} & \dots & \widetilde{d}_{{N+1}}\\
            d_3 &  -\widetilde{\lambda} & c_2  &  &   \\
            d_5 & a_2 & -\widetilde{\lambda} & \ddots  &   \\
            &   & \ddots & \ddots & c_N \\
            &   &  & a_N & -\widetilde{\lambda}   \\
        \end{array}
        \right| \\
        &=& (-\widetilde{\lambda} - u_m) \left( (-\widetilde{\lambda} + u_m)  \left|
        \begin{array}{cccc}
            -\widetilde{\lambda} & c_2  &  &   \\
            a_2 & -\widetilde{\lambda} & \ddots  &   \\
            & \ddots & \ddots & c_N \\
            &  & a_N & -\widetilde{\lambda}   \\
        \end{array}
        \right| -d_4  \left|
        \begin{array}{ccccc}
            d_2 & \widetilde{d}_{3} & \dots & \dots & \widetilde{d}_{{N+1}}\\
            a_2 &  -\widetilde{\lambda} & c_3  &  & \\
            & a_3 & -\widetilde{\lambda} & \ddots &\\
            & & \ddots & \ddots & c_N \\
            & &   & a_N & -\widetilde{\lambda}   \\
        \end{array}
        \right| \right) \\
        && - (d_1+\widetilde{d}_{{1}})  \left|
        \begin{array}{cccc}
            -\widetilde{\lambda} & c_2  &  &   \\
            a_2 & -\widetilde{\lambda} & \ddots  &   \\
            & \ddots & \ddots & c_N \\
            &  & a_N & -\widetilde{\lambda}   \\
        \end{array}
        \right| + d_3  \left|
        \begin{array}{ccccc}
            d_2 & \widetilde{d}_{3} & \dots & \dots & \widetilde{d}_{{N+1}}\\
            a_2 &  -\widetilde{\lambda} & c_3  &  & \\
            & a_3 & -\widetilde{\lambda} & \ddots &\\
            & & \ddots & \ddots & c_N \\
            & &   & a_N & -\widetilde{\lambda}   \\
        \end{array}
        \right| - d_5 \left|
        \begin{array}{cccccc}
            d_2 & \widetilde{d}_{3} & & \dots & & \widetilde{d}_{{N+1}}\\
            -\widetilde{\lambda} &  c_2 &  &  & & \\
            & a_3 & -\widetilde{\lambda} & c_4 & & \\
            & & a_4 & \ddots & \ddots & \\
            & & & \ddots & \ddots & c_N \\
            & & &  & a_N & -\widetilde{\lambda}   \\
        \end{array}
        \right|, 
    \end{eqnarray*}
    where the last term can be expressed as
    \begin{equation}
        \left|
        \begin{array}{cccccc}
            d_2 & \widetilde{d}_{3} & & \dots & & \widetilde{d}_{{N+1}}\\
            -\widetilde{\lambda} &  c_2 &  &  & & \\
            & a_3 & -\widetilde{\lambda} & c_4 & & \\
            & & a_4 & \ddots & \ddots & \\
            & & & \ddots & \ddots & c_N \\
            & & &  & a_N & -\widetilde{\lambda}   \\
        \end{array}
        \right| = d_{{2}} c_2 \left|
        \begin{array}{cccc}
            -\widetilde{\lambda} & c_4  &  &   \\
            a_4 & -\widetilde{\lambda} & \ddots  &   \\
            & \ddots & \ddots & c_N \\
            &  & a_N & -\widetilde{\lambda}   \\
        \end{array}
        \right|  + \widetilde{\lambda} \left|
        \begin{array}{ccccc}
            \widetilde{d}_{3} & \widetilde{d}_4 & \dots & \dots & \widetilde{d}_{N+1}\\
            a_3 &  -\widetilde{\lambda} & c_4  &  & \\
            & a_4 & -\widetilde{\lambda} & \ddots &\\
            & & \ddots & \ddots & c_N \\
            & &   & a_N & -\widetilde{\lambda}   \\
        \end{array}
        \right|.
    \end{equation}
    We now use the following definition from \cite{Koellermeier2020c} for the subdeterminant $ \left|A_i\right|$
    \begin{equation}\label{subdeterminant_A}
        \left|A_i\right| = \left|
        \begin{array}{cccc}
           -\widetilde{\lambda} & c_i  &  &   \\
          a_i & -\widetilde{\lambda} & \ddots  &   \\
       & \ddots & \ddots & c_N \\
         &  & a_N & -\widetilde{\lambda}   \\
        \end{array}
        \right|,
    \end{equation}
    and similarly we introduce a new definition for the subdeterminant $ \left|D_i\right|$
    \begin{equation}\label{subdeterminant_D}
        \left|D_i\right| = \left|
        \begin{array}{ccccc}
            \widetilde{d}_{i} & \widetilde{d}_{i+1} & \dots & \dots & \widetilde{d}_{N+1}\\
            a_i &  -\widetilde{\lambda} & c_{i+1}  &  & \\
            & a_{i+1} & -\widetilde{\lambda} & \ddots &\\
            & & \ddots & \ddots & c_N \\
            & &   & a_N & -\widetilde{\lambda}   \\
        \end{array}
        \right|.
    \end{equation}
    Note that from \eqref{subdeterminant_A} and \eqref{subdeterminant_D}, we directly get the recursion formulas
    \begin{equation*}
        \left|A_i\right| = \frac{1}{a_{i-2} c_{i-2}}\left( \left|A_{i-2}\right| + \widetilde{\lambda} \left|A_{i-1}\right| \right),
    \end{equation*}
    and
    \begin{equation*}
        \left|D_{i}\right| = \widetilde{d}_{{i}} \left|A_{i+1}\right| - a_{i}\left|D_{i+1}\right|, 
    \end{equation*}
    which also holds for $i=2$, when using $\widetilde{d}_{{2}}= d_2$.

    Using both definitions, we can first simplify the characteristic polynomial and then insert the recursion formula for $\left|D_2\right|$ to obtain
    \begin{eqnarray*}
        \left| \widetilde{A}^{\textrm{MHSWME}} - \widetilde{\lambda} I_{N+2} \right| &=& (-\widetilde{\lambda} - u_m)  \left( (-\widetilde{\lambda} + u_m)  \left|A_2\right| - d_4 \left( d_{2} \left|A_{3}\right| - a_{2}\left|D_{3}\right| \right) \right) \\
        && - (d_1+\widetilde{d}_{{1}})  \left|A_2\right| + d_3  \left( d_{2} \left|A_{3}\right| - a_{2}\left|D_{3}\right| \right) - d_5 \left( d_{2} c_2  \left|A_4\right|  + \widetilde{\lambda} \left|D_3\right| \right) \\
        &=& (-\widetilde{\lambda} - u_m)  \left( (-\widetilde{\lambda} + u_m)  \left|A_2\right| - d_2 d_4  \left|A_3\right| \right) - d_1 \left|A_2\right| + d_2 \left( d_3 \left|A_3\right| - c_2 d_5 \left|A_4\right| \right) \\
        && + (-\widetilde{\lambda} - u_m) d_4 a_2 \left|D_{3}\right| - \widetilde{d}_{{1}}  \left|A_2\right| + d_3 a_2 \left|D_{3}\right| - d_5 \widetilde{\lambda} \left|D_{3}\right|\\
        &=& \left| \widetilde{A}^{\textrm{MHSWME}} - \widetilde{\lambda} I_{N+2} \right|_1 + \left| \widetilde{A}^{\textrm{MHSWME}} - \widetilde{\lambda} I_{N+2} \right|_2\\
    \end{eqnarray*}
    where $\left| \widetilde{A}^{\textrm{MHSWME}} - \widetilde{\lambda} I_{N+2} \right|_1= \left| \widetilde{A}^{\textrm{HSWME}} - \widetilde{\lambda} I_{N+2} \right|$ was shown in Theorem 3.2 in \cite{Koellermeier2020c} to be equal to $\left|A_2\right| \left( \left(\lambda - u_m\right)^2 -gh -\alpha_1^2 \right)$, and $\left| \widetilde{A}^{\textrm{MHSWME}} - \widetilde{\lambda} I_{N+2} \right|_2$ can be computed as
    \begin{eqnarray*}
        \left| \widetilde{A}^{\textrm{MHSWME}} - \widetilde{\lambda} I_{N+2} \right|_2 &=& - \widetilde{d}_{{1}}  \left|A_2\right| + \left((-\widetilde{\lambda} - u_m) d_4 a_2  + d_3 a_2 - d_5 \widetilde{\lambda} \right) \left|D_{3}\right|\\
        &=& \sum\limits_{i=2}^N \frac{\alpha_i^2}{2i+1} \left|A_2\right| + \left( -\frac{2}{3} \widetilde{\lambda} - \frac{2}{3} u_m + \frac{2}{3}u_m + \frac{2}{3}\widetilde{\lambda} \right) \left|D_{3}\right|\\
        &=& \sum\limits_{i=2}^N \frac{\alpha_i^2}{2i+1} \left|A_2\right|.
    \end{eqnarray*}
    This means that the characteristic polynomial is given by
    \begin{equation}
        \left| \widetilde{A}^{\textrm{MHSWME}} - \widetilde{\lambda} I_{N+2} \right| = \left|A_2\right| \left( \left(\lambda - u_m\right)^2 -gh - \alpha_1^2 + \sum\limits_{i=2}^N \frac{\alpha_i^2}{2i+1}\right).
    \end{equation}
    The eigenvalues follow directly.
\end{proof}

From Theorem \ref{th:hybrid_convective_system} it is obvious that the MHSWME is not globally hyperbolic, since the eigenvalues $\lambda_{N+1,N+2} = u_m \pm \sqrt{gh + \alpha_1^2 - \sum\limits_{i=2}^N\frac{\alpha_i^2}{2i+1}}$ become complex-valued for $gh + \alpha_1^2 < \sum\limits_{i=2}^N\frac{\alpha_i^2}{2i+1}$. This can be the case for extreme non-equilibrium with large values of $\alpha_i$, for $i\geq 2$. However, in contrast to the SWME from \cite{Kowalski2019}, the hyperbolicity region does not shrink with increasing $N$. This is illustrated in Figure \ref{fig:HybridHypRegion}, plotting the hyperbolicity regions for MHSWME with $N=2$ and $N=3$. There is always a large hyperbolicity region around the origin at $\alpha_i = 0$, that does not shrink in size even as $N$ increases, as can also be seen from the formula of the eigenvalues $\lambda_{N+1,N+2}$. The MHSWME is thus only locally hyperbolic, but with reasonably large hyperbolicity region.

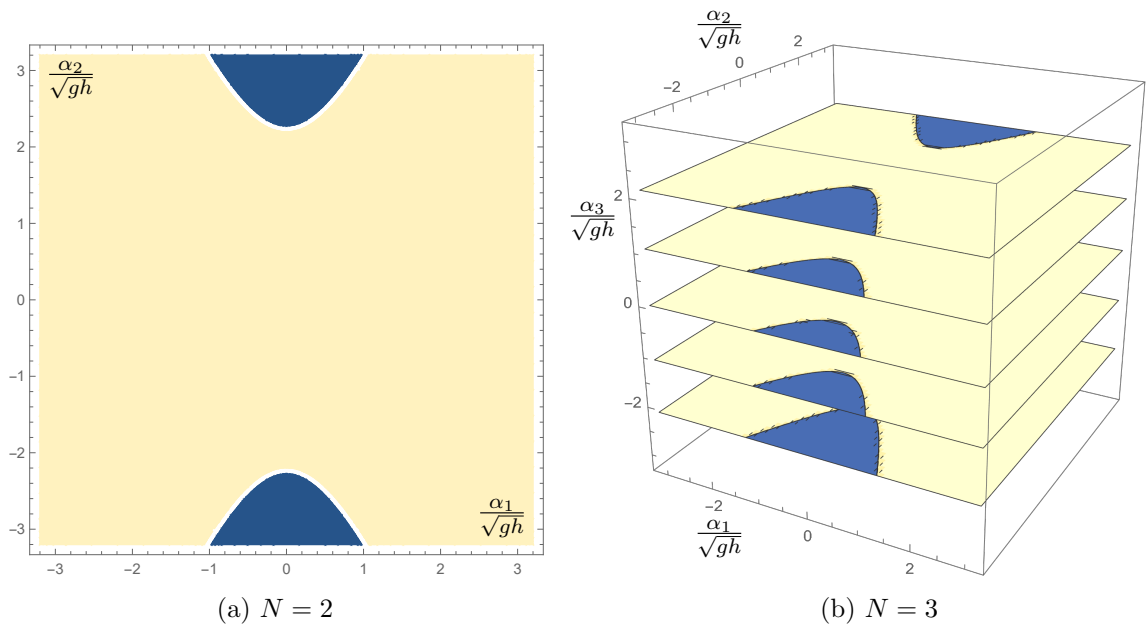
\begin{figure}[h!]
    \centering
    \begin{subfigure}[b]{0.47\textwidth}
        \centering
        \begin{overpic}[width=1\linewidth]{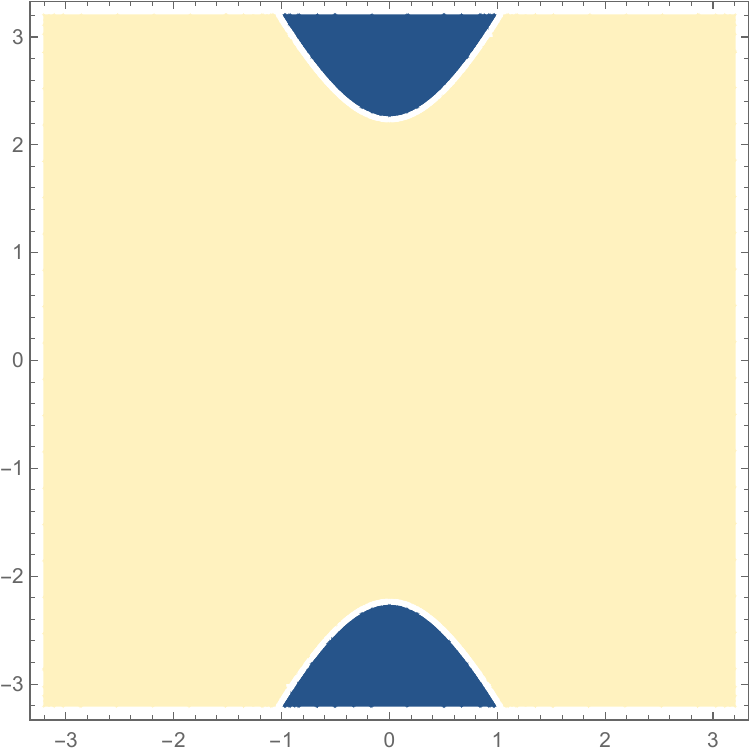}
        \put(87,11){$\frac{\alpha_1}{\sqrt{gh}}$}
        \put(7,93){$\frac{\alpha_2}{\sqrt{gh}}$}
        \end{overpic}
        \caption{$N=2$}
        \label{fig:HybridHypRegionN2}
    \end{subfigure}
    \hfill
    \begin{subfigure}[b]{0.47\textwidth}
        \centering
        \begin{overpic}[width=1\linewidth]{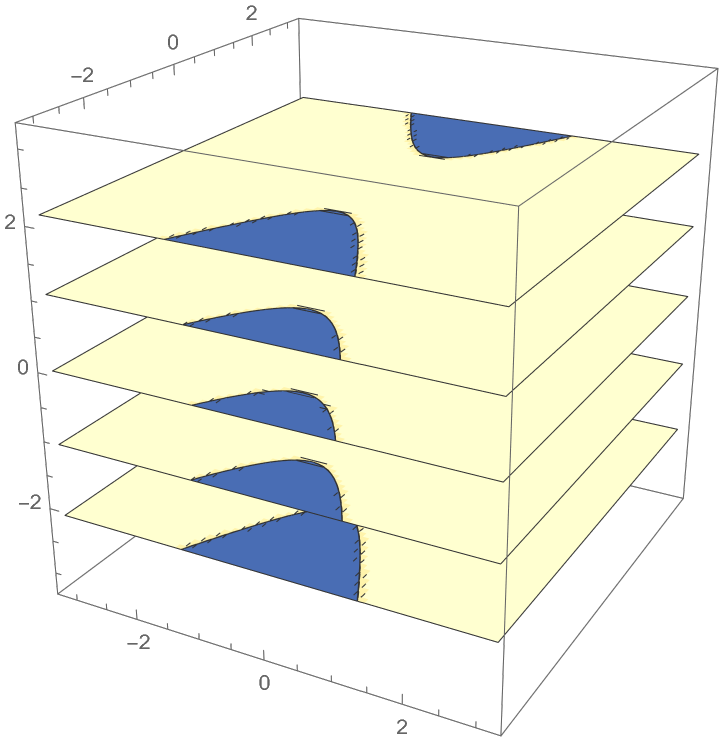}
        \put(15,7){$\frac{\alpha_1}{\sqrt{gh}}$}
        \put(15,100){$\frac{\alpha_2}{\sqrt{gh}}$}
        \put(-8,65){$\frac{\alpha_3}{\sqrt{gh}}$}
        \end{overpic}
        \caption{$N=3$}
        \label{fig:HybridHypRegionN3}
    \end{subfigure}
    \caption{Hyperbolicity regions of MHSWME for $N=2$ (left) depending on scaled $\alpha_1, \alpha_2$ and $N=3$ (right) depending on scaled $\alpha_1, \alpha_2, \alpha_3$. MHSWME shows large hyperbolicity regions that do not decrease in size with increasing $N$.}
    \label{fig:HybridHypRegion}
\end{figure}

As the higher-order moment equations of the MHSWME are the same as those for HSWME, the analytical computation of steady states is equally infeasible, posing a disadvantage of the MHSWME.

\section{Primitive regularization for Hyperbolic SWME (PHSWME)}
\label{sec:new_hyp}
The hyperbolic regularization to obtain the HSWME (as well as the MHSWME) is performed using the convective variables $U_{c}$ \eqref{eq:SWME_c}. It effectively uses vanishing higher-order coefficients $\alpha_i = 0$ for $i \geq 2$ for the evaluation of the convective system matrix.
The idea of that hyperbolic regularization originated from kinetic theory, as a similar regularization also led to a (not very accurate, but hyperbolic) set of moment equations as shown in \cite{Koellermeier2015}. In the same paper, it was also shown that another hyperbolic regularization from \cite{Cai2013b} was the same as simply linearizing a transformed system matrix. Equipped with that knowledge, we apply a similar idea following five steps:
\begin{enumerate}
    \item transform the convective system matrix to primitive variables, see Lemma \ref{lemma:SWME_p}
    \item regularize the primitive system matrix
    \item investigate hyperbolicity of the regularized primitive system matrix
    \item transform the regularized primitive system matrix back to convective variables
    \item conclude that eigenvalues are preserved by the similarity transformation from primitive to convective variables
\end{enumerate}

The first step was already performed in Lemma \ref{lemma:SWME_p}, so that we directly continue with the second step, the regularization.

The next lemma states the regularization of the primitive system matrix, leading to the new so-called Primitive regularization for Hyperbolic Shallow Water Moment Equations (PHSWME).
\begin{lemma}\label{lemma:APHSWME_p}
    Regularizing the primitive system matrix \eqref{eq:A_SWME_p} around linear velocity profiles, i.e., evaluating at $U_{p}=\left(h,u_m,\alpha_1, 0, \ldots, 0\right)$, the Primitive regularization for Hyperbolic Shallow Water Moment Equations (PHSWME) are obtained as
    \begin{equation}\label{eq:PHSWME_p}
        \frac{\partial U_{p}}{\partial t} + A^{\textrm{PHSWME}}_{p} \frac{\partial U_{p}}{\partial x} = 0,
    \end{equation}
    using the primitive PHSWME system matrix $A^{\textrm{PHSWME}}_{p} \in \mathbb{R}^{(N+2)\times(N+2)}$ defined as
    \begin{equation}\label{eq:A_PHSWME_p}
        A^{\textrm{PHSWME}}_{p} =
        \left(\begin{array}{cccccc}
        u_m & h &  &  &  &  \\
        g + \frac{\alpha_1^2}{3 h} & u_m & \frac{2}{3} \alpha_1 &  &  &  \\
         & \alpha_1 & u_m & \frac{3}{5} \alpha_1 & &\\
        -\frac{\alpha_1^2}{3h} &   & \frac{1}{3}\alpha_1 & u_m & \ddots & \\
         &  &  & \ddots & \ddots & \frac{N+1}{2N+1}\alpha_1 \\
         &  &  &  & \frac{N-1}{2N-1}\alpha_1 & u_m
        \end{array}\right).
    \end{equation}
\end{lemma}
\begin{proof}
    From applying the regularization $\alpha_i = 0$ for all $i\geq2$, to the primitive system matrix in \eqref{eq:A_SWME_p} we directly get
    \begin{equation}\label{eq:A_lin_p}
        A^{\textrm{PHSWME}}_{p} =
        \left(\begin{array}{cccccc}
        u_m & h &   &   &   &   \\
        g + \frac{\alpha_1^2}{3 h} & u_m & \frac{2}{3} \alpha_1 &   &   &  \\
        \frac{1}{h}\left(B_{111} + A_{111} \right) \alpha_1^2 & \alpha_1 &  & & &\\
        \vdots &  & & \mathcal{A}^{lin} & & \\
        \frac{1}{h}\left(B_{N11} + A_{N11} \right) \alpha_1^2 &  &  & & &
        \end{array}\right).
    \end{equation}
    The entries $2+i=3, \ldots, N$ of the first column can be simplified as $\frac{1}{h}\left(B_{i11} + A_{i11} \right) \alpha_1^2 = -\frac{\alpha_1^2}{3h} \delta_{i,2}$, leaving only an entry for $i=2$, i.e., in the fourth equation.

    The lower right block matrix is given by $\mathcal{A}^{lin}_{i,l} = \left( B_{il1} + 2 A_{i1l} \right) \alpha_1 + u_m \delta_{i,l}$. These entries are the same as in the lower right block of the convective HSWME system matrix $A^{\textrm{HSWME}}_{c}$ derived in \cite{Koellermeier2020c} as $\mathcal{A}^{lin} = u_m \cdot I_N + A_2$ with the matrix $A_2 \in \mathbb{R}^{N \times N}$ defined also in \eqref{eq:A_2}. The derivation of the elements in $A_2$ is given in \cite{Koellermeier2020c}, Appendix A. This completes the proof.
\end{proof}

Using the explicit definition of the system matrix in primitive variables after the regularization, we can prove hyperbolicity of the PHSWME by computing eigenvalues via the characteristic polynomial. Note that there are two possibilities: (1) compute the eigenvalues of $A^{\textrm{PHSWME}}_{p}$ directly, or (2) compute the eigenvalues of the matrix transformed back to convective variables. In this paper, we will do the former because the matrix $A^{\textrm{PHSWME}}_{p}$ is more sparse than its convective counterpart, as we will later see.

Using Lemma \ref{lemma:A_2_chi}, the characteristic polynomial of the matrix $A^{\textrm{PHSWME}}_{p}$ in \eqref{eq:A_PHSWME_p} can be computed.

\begin{theorem}\label{th:APHSWME_p_chi}
    The matrix $A^{\textrm{PHSWME}}_{p} \in \mathbb{R}^{(N+2)\times(N+2)}$ defined in \eqref{eq:A_PHSWME_p} has the following characteristic polynomial
    \begin{equation}\label{eq:APHSWME_p_char_poly}
        \chi_{A^{\textrm{PHSWME}}_{p}}(\lambda) = \frac{(-\alpha_1)^NN!}{(2N+1)!!} \cdot P^{'}_{N+1}\left(\frac{\lambda -u_m}{\alpha_1}\right) \cdot \left( \left( \lambda - u_m \right)^2 - gh - \alpha_1^2 \right),
    \end{equation}
    and the real and pairwise distinct eigenvalues $\lambda_1, \ldots, \lambda_{N+2}$ of $A^{\textrm{PHSWME}}_{p}$ are given by
    \begin{eqnarray}\label{eq:APHSWME_p_EV}
        \lambda_i, i=1,\ldots, N: P'_{N+1}\left(\frac{\lambda_i - u_m}{\alpha_1}\right)=0,
        \lambda_{N+1,N+2} = u_m \pm \sqrt{gh + \alpha_1^2}.
    \end{eqnarray}
\end{theorem}
\begin{proof}
    The proof is similar to the proof of Theorem \ref{th:hybrid_convective_system} and Theorem 3.1 from \cite{Koellermeier2020c} and uses Lemma \ref{lemma:A_2_chi}. It follows from direct development of the determinant with respect to the first two rows, with subsequent use of recurrence formulas for the system matrix.

    We start by defining
    \begin{equation}
       A^{\textrm{PHSWME}}_{p} = \widetilde{A}^{{\textrm{PHSWME}}} + u_m I_{N+2}, \quad \widetilde{\lambda} = \lambda - u_m,
    \end{equation}
    such that the characteristic polynomial of $A^{\textrm{PHSWME}}_{p}$ can be computed by
    \begin{eqnarray}
       \chi_{A^{\textrm{PHSWME}}_{p}}(\lambda) &=& \det\left( A^{\textrm{PHSWME}}_{p} - \lambda I_{N+2} \right) = \det\left( \widetilde{A}^{{\textrm{PHSWME}}} - \left(\lambda - u_m \right) I_{N+2} \right) \\
       &=& \det\left( \widetilde{A}^{{\textrm{PHSWME}}} - \widetilde{\lambda} I_{N+2} \right) =: \left| \widetilde{A}^{{\textrm{PHSWME}}} - \widetilde{\lambda} I_{N+2} \right|,
    \end{eqnarray}
    with $\widetilde{A}^{{\textrm{PHSWME}}} =$
    \begin{equation}
        \left(\begin{array}{cccccc}
        - \widetilde{\lambda} & h &  &  &  &  \\
        g + \frac{\alpha_1^2}{3 h} & - \widetilde{\lambda} & \frac{2}{3} \alpha_1 &  &  &  \\
         & \alpha_1 & - \widetilde{\lambda} & \frac{3}{5} \alpha_1 & &\\
        -\frac{\alpha_1^2}{3h} &   & \frac{1}{3}\alpha_1 & - \widetilde{\lambda} & \ddots & \\
         &  &  & \ddots & \ddots & \frac{N+1}{2N+1}\alpha_1 \\
         &  &  &  & \frac{N-1}{2N-1}\alpha_1 & - \widetilde{\lambda}
        \end{array}\right) =
        \left(\begin{array}{cccccc}
        - \widetilde{\lambda} & d_0 &  &  &  &  \\
        d_1 & - \widetilde{\lambda} & d_2 &  &  &  \\
        0 & d_4 & - \widetilde{\lambda} & c_2 & &\\
        d_5 &   & a_2 & - \widetilde{\lambda} & \ddots & \\
         &  &  & \ddots & \ddots & c_N \\
         &  &  &  & a_N & - \widetilde{\lambda}
        \end{array}\right),
    \end{equation}
    where we made the following definitions
    \begin{equation}
       d_0 = h, \quad d_1 = g + \frac{\alpha_1^2}{3h}, \quad d_2 = \frac{2}{3}\alpha_1, \quad d_4 = \alpha_1, \quad d_5 = -\frac{\alpha_1^2}{3h},
    \end{equation}
    and as usual we denote $a_i = \frac{i-1}{2i-1}\alpha_1$ and $c_i = \frac{i+1}{2i+1}\alpha_1$, for $i=2,\ldots, N$.
    Note that this means that the proof will be similar to the proof of Theorem \ref{th:hybrid_convective_system} and Theorem 3.1 from \cite{Koellermeier2020c}, but with different definitions of the entries $d_i$ due to the setting in primitive variables.

    Now we compute $\left| \widetilde{A}^{{\textrm{PHSWME}}} - \widetilde{\lambda} I_{N+2} \right|$ by always developing with respect to the first row.
    \begin{eqnarray*}
        \left| \widetilde{A}^{{\textrm{PHSWME}}} - \widetilde{\lambda} I_{N+2} \right| &=& \left|
        \begin{array}{cccccc}
        -\widetilde{\lambda}  & d_0 &  &  &  & \\
        d_1 & -\widetilde{\lambda} & d_2 &  & & \\
         0 & d_4 & -\widetilde{\lambda} & c_2  &  &   \\
         d_5 &  & a_2 & -\widetilde{\lambda} & \ddots  &   \\
        &  &  & \ddots & \ddots & c_N \\
         &  &  &  & a_N & -\widetilde{\lambda}  \\
        \end{array}
        \right| \\
      &=& (-\widetilde{\lambda})
      \left|
        \begin{array}{ccccc}
         -\widetilde{\lambda} & d_2 &  & & \\
          d_4 & -\widetilde{\lambda} & c_2  &  &   \\
           & a_2 & -\widetilde{\lambda} & \ddots  &   \\
        &  & \ddots & \ddots & c_N \\
         &  &  & a_N & -\widetilde{\lambda}   \\
        \end{array}
        \right|
        -d_0 \left|
        \begin{array}{ccccc}
        d_1 & d_2 &  &  & \\
         0 &  -\widetilde{\lambda} & c_2  &  &   \\
         d_5 & a_2 & -\widetilde{\lambda} & \ddots  &   \\
        &   & \ddots & \ddots & c_N \\
         &   &  & a_N & -\widetilde{\lambda}   \\
        \end{array}
        \right| \\
        &=& (-\widetilde{\lambda}) \left( (-\widetilde{\lambda})  \left|
        \begin{array}{cccc}
           -\widetilde{\lambda} & c_2  &  &   \\
          a_2 & -\widetilde{\lambda} & \ddots  &   \\
       & \ddots & \ddots & c_N \\
         &  & a_N & -\widetilde{\lambda}   \\
        \end{array}
        \right| -d_2  \left|
        \begin{array}{ccccc}
          d_4 & c_2  &  &   \\
           &  -\widetilde{\lambda} & \ddots  &   \\
        & \ddots & \ddots & c_N \\
         &   & a_N & -\widetilde{\lambda}   \\
        \end{array}
        \right| \right) \\
        && - d_0 \left( d_1  \left|
        \begin{array}{cccc}
           -\widetilde{\lambda} & c_2  &  &   \\
          a_2 & -\widetilde{\lambda} & \ddots  &   \\
        & \ddots & \ddots & c_N \\
         &  & a_N & -\widetilde{\lambda}   \\
        \end{array}
        \right| - d_2  \left|
        \begin{array}{cccc}
         0  & c_2  &  &   \\
         d_5  & -\widetilde{\lambda} & \ddots  &   \\
        &  \ddots & \ddots & c_N \\
         &   & a_N & -\widetilde{\lambda}   \\
        \end{array}
        \right|\right) \\
        &=& (-\widetilde{\lambda})  \left( (-\widetilde{\lambda})  \left|A_2\right| - d_2 d_4  \left|A_3\right| \right) - d_0 \left(d_1  \left|A_2\right| - d_2 \left( - c_2 d_5 \left|A_4\right| \right)\right),
    \end{eqnarray*}
    where the $|A_i|$ were already introduced in Theorem \ref{th:hybrid_convective_system} as the subdeterminants     \begin{equation}\label{eq:subdeterminant}
       |A_i| = \left|
        \begin{array}{cccc}
           -\widetilde{\lambda} & c_i  &  &   \\
          a_i & -\widetilde{\lambda} & \ddots  &   \\
       & \ddots & \ddots & c_N \\
         &  & a_N & -\widetilde{\lambda}   \\
        \end{array}
        \right|,
    \end{equation}
    with recursion formula
    \begin{equation*}
        \left|A_i\right| = \frac{1}{a_{i-2} c_{i-2}}\left( \left|A_{i-2}\right| + \widetilde{\lambda} \left|A_{i-1}\right| \right).
    \end{equation*}
    We can then eliminate $\left|A_4\right|$ from the expression and factorize the remaining terms with respect to $\left|A_2\right|$ and $\left|A_3\right|$ as follows
    \begin{eqnarray*}
       \chi_{A^{\textrm{PHSWME}}_{p}}(\lambda) &=& \widetilde{\lambda}^2 |A_2| + d_2 d_4 \widetilde{\lambda} |A_3| - d_0 d_1 |A_2| + d_0 d_2 c_2 d_5 |A_4|\\
       &=& \widetilde{\lambda}^2 |A_2| + d_2 d_4 \widetilde{\lambda} |A_3| - d_0 d_1 |A_2| + \frac{d_0 d_2 d_5}{a_2}\left( \left|A_2\right| + \widetilde{\lambda} \left|A_3\right| \right)\\
        &=& |A_2| \left( \widetilde{\lambda}^2 - d_0 d_1 + \frac{d_0 d_5 d_2}{a_2} \right) + |A_3| \widetilde{\lambda} \left( d_2 d_4 + \frac{d_0 d_5 d_2}{a_2} \right)\\
        &=& \left|A_2\right| \left( \widetilde{\lambda}^2 -g h -\alpha_1^2 \right) + \left|A_3\right| \widetilde{\lambda} \left( 0 \right),
    \end{eqnarray*}
    where we inserted the specific terms for the variables $d_i$ and $a_2$.
    Going back to the standard notation $\widetilde{\lambda} = \lambda - u_m$ and using $|A_2| = \chi_{A_2}(\widetilde{\lambda})$ we finally obtain
    \begin{equation*}
        \chi_{A^{\textrm{PHSWME}}_{p}}(\lambda) = \left( \left(\lambda - u_m\right)^2 -gh -\alpha_1^2 \right) \cdot \chi_{A_2} (\lambda - u_m),
    \end{equation*}
    so that we can directly use the characteristic polynomial $\chi_{A_2}$ of Lemma \ref{lemma:A_2_chi},
    which completes the proof.
\end{proof}

We note that $A^{\textrm{PHSWME}}_{p}$ is also real diagonalizable for the limiting case $\alpha_1 = 0$.
This means that the PHSWME system is indeed hyperbolic.

Note that the eigenvalues only depend on the first moment coefficient $\alpha_1$ and not on the higher-order coefficients.\\

In the fourth step, transform the system matrix back to the convective variables.

\begin{lemma}\label{lemma:APHSWME_c}
    The PHSWME system defined in \eqref{eq:A_PHSWME_p} can be transformed from primitive to convective variables to read
    \begin{equation}\label{eq:PHSWME_c}
        \frac{\partial U_{c}}{\partial t} + A^{\textrm{PHSWME}}_{c} \frac{\partial U_{c}}{\partial x} = 0,
    \end{equation}
    using the convective PHSWME system matrix $A^{\textrm{PHSWME}}_{c} \in \mathbb{R}^{(N+2)\times(N+2)}$ defined as $A^{\textrm{PHSWME}}_{c}=$
    \begin{equation}\label{eq:A_PHSWME_c}
        \left(\begin{array}{cccccccc}
          & 1 &  &  &  &  &  &  \\
        -u_m^2 + gh - \frac{\alpha_1^2}{3} & 2u_m & \frac{2}{3} \alpha_1 &  &  &  &  &  \\
        -2u_m \alpha_1 - \frac{3}{5}\alpha_1\alpha_2 & 2\alpha_1 & u_m & \frac{3}{5} \alpha_1 & &&  &  \\
        -u_m \alpha_2 - \frac{4}{7}\alpha_1\alpha_3 - \frac{2}{3}\alpha_1^2 & \alpha_2 & \frac{1}{3}\alpha_1 & u_m & \ddots & &  &  \\
        \vdots & \vdots &  & \ddots & \ddots & \frac{i+1}{2i+1}\alpha_1 &  &   \\
        -u_m \alpha_i- \frac{i-1}{2i-1}\alpha_1\alpha_{i-1} - \frac{i+1}{2i+1}\alpha_1 \alpha_{i+1}&  \alpha_i &  &  & \frac{i-1}{2i-1}\alpha_1 & u_m & \ddots &  \\
        \vdots & \vdots &  &  &  & \ddots & \ddots &  \frac{N+1}{2N+1}\alpha_1 \\
        -u_m \alpha_N - \frac{N-1}{2N-1} \alpha_1 \alpha_N & \alpha_N &  &  &  &  &  \frac{N-1}{2N-1}\alpha_1 & u_m
        \end{array}\right).
    \end{equation}
\end{lemma}
\begin{proof}
    We start from the primitive system matrix $A^{\textrm{PHSWME}}_{p}$ \eqref{eq:A_PHSWME_p} and compute the convective system matrix using the variable transformation $A^{\textrm{PHSWME}}_{c} = \frac{\partial  T\left(U_{p}\right)}{\partial U_{p}} A^{\textrm{PHSWME}}_{p} \frac{\partial T\left(U_{p}\right)}{\partial U_{p}}^{-1}$ with the transformation matrices given in \eqref{eq:dervars} and \eqref{eq:inv_dervars}. The first step reads
    \begin{equation}
        \frac{\partial  T\left(U_{p}\right)}{\partial U_{p}} A^{\textrm{PHSWME}}_{p} =
        h\left(\begin{array}{cccccccc}
        \frac{u_m}{h} & 1 &  &  &  & &  & \\
        \frac{u_m^2}{h} + g + \frac{\alpha_1^2}{3 h} & 2u_m & \frac{2}{3} \alpha_1 &  &  & &  & \\
        \frac{u_m \alpha_1}{h} & 2\alpha_1 & u_m & \frac{3}{5} \alpha_1 & &&  &\\
        \frac{u_m \alpha_2}{h}-\frac{\alpha_1^2}{3h} & \alpha_2  & \frac{1}{3}\alpha_1 & u_m & \ddots & &  &\\
        \vdots & \vdots &  & \ddots & \ddots & \frac{i+1}{2i+1}\alpha_1 &  &\\
        \frac{u_m \alpha_i}{h} & \alpha_i &  &  & \frac{i-1}{2i-1}\alpha_1 & u_m & \ddots & \\
        \vdots & \vdots &  &  &  & \ddots & \ddots & \frac{N+1}{2N+1}\alpha_1\\
        \frac{u_m \alpha_N}{h} & \alpha_N &  &  & &  &\frac{N-1}{2N-1}\alpha_1 & u_m
        \end{array}\right),
    \end{equation}
    so that we can directly compute in a second step that we have indeed $A^{\textrm{PHSWME}}_{c} = \left( \frac{\partial  T\left(U_{p}\right)}{\partial U_{p}} A^{\textrm{PHSWME}}_{p} \right) \frac{\partial T\left(U_{p}\right)}{\partial U_{p}}^{-1} = $
    \begin{equation}
        \left(\begin{array}{cccccccc}
          & 1 &  &  &  &  &  &  \\
        -u_m^2 + gh - \frac{\alpha_1^2}{3} & 2u_m & \frac{2}{3} \alpha_1 &  &  &  &  &  \\
        -2u_m \alpha_1 - \frac{3}{5}\alpha_1\alpha_2 & 2\alpha_1 & u_m & \frac{3}{5} \alpha_1 & &&  &  \\
        -u_m \alpha_2 - \frac{4}{7}\alpha_1\alpha_3 - \frac{2}{3}\alpha_1^2 & \alpha_2 & \frac{1}{3}\alpha_1 & u_m & \ddots & &  &  \\
        \vdots & \vdots &  & \ddots & \ddots & \frac{i+1}{2i+1}\alpha_1 &  &   \\
        -u_m \alpha_i- \frac{i-1}{2i-1}\alpha_1\alpha_{i-1} - \frac{i+1}{2i+1}\alpha_1 \alpha_{i+1}&  \alpha_i &  &  & \frac{i-1}{2i-1}\alpha_1 & u_m & \ddots &  \\
        \vdots & \vdots &  &  &  & \ddots & \ddots &  \frac{N+1}{2N+1}\alpha_1 \\
        -u_m \alpha_N - \frac{N-1}{2N-1} \alpha_1 \alpha_N & \alpha_N &  &  &  &  &  \frac{N-1}{2N-1}\alpha_1 & u_m
        \end{array}\right).
    \end{equation}
\end{proof}

While the regularization in primitive variables sets $\alpha_i= 0$ for $i\geq2$ in $A^{\textrm{PHSWME}}_{p}$, the transformation back to convective variables makes the coefficients $\alpha_i$ reappear in the convective system matrix $A^{\textrm{PHSWME}}_{c}$, see \eqref{eq:A_PHSWME_c}.

We note that the HSWME system matrix $A^{\textrm{HSWME}}_{c}$ in \eqref{eq:A_HSWME_c} from \cite{Koellermeier2020c} is recovered from \eqref{eq:PHSWME_c} when again setting $\alpha_i= 0$ for $i\geq2$ in $A^{\textrm{PHSWME}}_{c}$.
We observe that the differences between $A^{\textrm{PHSWME}}_{c}$ and $A^{\textrm{HSWME}}_{c}$ are in the first two columns, marking the crucial dependence of the moment equations for $h \alpha_i$, $i\geq2$ on the derivatives of the macroscopic water height $h$ and the water discharge $h u_m$.

We can conclude that the eigenvalues of the convective PHSWME system matrix \eqref{eq:A_PHSWME_c} are the same as the eigenvalues of the primitive PHSWME system matrix \eqref{eq:A_PHSWME_p}, due to the similarity transformation between the different systems. The global hyperbolicity of the convective system \eqref{eq:PHSWME_c} is thus proven by Theorem \ref{th:APHSWME_p_chi}.\\

Besides hyperbolicity, the new PHSWME system also allows for an analytical identification of steady states. This topic was first discussed in \cite{Koellermeier2020i}, where it was shown that the original HSWME system does not allow for analytical identification of steady states for $N\geq2$. The results for the new primitive regularization are given in the following theorem.

\begin{theorem}\label{th:APHSWME_steady_states}
    The PHSWME system defined in \eqref{eq:A_PHSWME_p} allows for the following steady states
    \begin{equation}\label{eq:APHSWME_steady_states}
        h=h_0 \quad \textrm{or} \quad -\Fr^2 + \frac{1}{2} \left(\left(\frac{h}{h_0}\right)^2+\left(\frac{h}{h_0}\right)\right) + \frac{\Fr^2 \Ma_1^2}{3}  \left( \left(\frac{h}{h_0}\right)^3+\left(\frac{h}{h_0}\right)^2+\left(\frac{h}{h_0}\right)\right)=0,
    \end{equation}
    where $\Fr= \frac{u_{m,0}}{\sqrt{gh_0}}$ denotes the dimensionless Froude number at a reference state and $\Ma_1 = \frac{\alpha_{1,0}}{u_{m,0}}$ denotes the dimensionless first moment number at the reference state.
\end{theorem}
\begin{proof}
    For the steady states, we start from the PHSWME system written in convective variables \eqref{eq:PHSWME_c}. From there on, we separately discuss three parts:

    \begin{itemize}
    \item[$(1)$] We directly see that the first equation leads to
    \begin{equation}
        \partial_x\left( h u_m \right) = 0
    \end{equation}
    \item[$(2)$] The second equation in steady state reads
    \begin{align}
        &&\quad\left(gh - u_m^2 - \frac{1}{3}\alpha_1^2\right)\partial_x h + 2 u_m \partial_x\left( h u_m \right) + \frac{2}{3} \alpha_1 \partial_x\left( h \alpha_1 \right) &= 0 \\
        \Rightarrow\quad && \partial_x \left( \frac{1}{2}gh^2 \right) - u_m^2 \frac{h u_m}{h u_m} \partial_x h - \frac{1}{3} \alpha_1^2 \partial_x h + \frac{2}{3} \alpha_1^2 \partial_x h + \frac{2}{3} \alpha_1 h \partial_x \alpha_1 &= 0 \\
        \Rightarrow\quad && \partial_x \left( \frac{1}{2}gh^2 \right) - u_m^3 h \partial_x \left( \frac{1}{u_m} \right) + \frac{1}{3} \alpha_1^2 \partial_x h + \frac{1}{3} h \partial_x \alpha_1 &= 0 \\
        \Rightarrow\quad && \partial_x \left( \frac{1}{2}gh^2 \right) + \partial_x \left( h u_m^2 \right) + \partial_x \left( \frac{1}{3} \alpha_1^2 \right) &= 0 \\
        \Rightarrow\quad && \partial_x \left( \frac{1}{2}gh^2 + h u_m^2 + \frac{1}{3} \alpha_1^2 \right) &= 0
    \end{align}
    \item[$(2+i)$] For equation $2+i$, with $i=1,\ldots, N$, we have in steady state that
    \begin{eqnarray}
        \left(-u_m \alpha_i - \frac{2}{3} \alpha_1 \delta_{i,2}- \frac{i-1}{2i-1}\alpha_1 \alpha_i - \frac{i+1}{2i+1} \alpha_i \alpha_{i+1} \right)\partial_x h + \alpha_i \partial_x\left( h u_m \right) \nonumber \\
        + \frac{i-1}{2i-1} \alpha_1 \partial_x\left( h \alpha_{i-1} \right) + u_m \partial_x\left( h \alpha_{i} \right) + \frac{i+1}{2i+1} \alpha_1 \partial_x\left( h \alpha_{i+1} \right)= 0
    \end{eqnarray}
    From these equations, we can derive the following set of equations
    \begin{equation}
        \left(\begin{array}{cccc}
        u_m & \frac{3}{5}\alpha_1 & &  \\
        \frac{1}{3}\alpha_1 & u_m & \ddots &  \\
         & \ddots & \ddots & \frac{N+1}{2N+1}\alpha_1 \\
         &  & \frac{N-1}{2N-1}\alpha_1 & u_m \\
         \end{array}\right) \cdot \left(\begin{array}{cccc}
        h^2 &  & &  \\
         & h &  &  \\
         & & \ddots &  \\
         &  &  & h \\
         \end{array}\right) \left(\begin{array}{c}
         \partial_x \left( \frac{\alpha_1}{h} \right) \\
         \partial_x \left( \alpha_2\right) \\
         \vdots \\
         \partial_x \left( \alpha_N\right) \\
         \end{array}\right) = \left(\begin{array}{c}
         0 \\
         \vdots \\
         0 \\
         \end{array}\right)
    \end{equation}
    We recognize that the left hand side includes precisely the matrix $A_2 + u_m \cdot I_N$ from \eqref{eq:A_2}. For $h\neq 0$ and non-zero velocity profile (i.e. either $u_m \neq 0$ or $\alpha_1 \neq 0$), the left hand side matrix is non-singular. This means that the only non-trivial solution for steady states $h \neq 0$ is given by the kernel of $A_2 + u_m \cdot I_N$, i.e.,
    \begin{equation}
        \partial_x \left( \frac{\alpha_1}{h} \right) = 0 \quad \textrm{and} \quad \partial_x  \alpha_i = 0,  \textrm{ for } i=2,\ldots, N.
    \end{equation}
    \end{itemize}
    Together, we then have the following conditions in steady state
    \begin{eqnarray}
        \partial_x\left( h u_m \right) = 0 \Longrightarrow &h u_m  = \textrm{const} \\
        \partial_x \left( \frac{1}{2}gh^2 + h u_m^2 + \frac{1}{3} \alpha_1^2 \right) = 0 \Longrightarrow &\frac{1}{2}gh^2 + h u_m^2 + \frac{1}{3} \alpha_1^2 = \textrm{const} \\
        \partial_x \left( \frac{\alpha_1}{h} \right) = 0 \Longrightarrow & \frac{\alpha_1}{h} = \textrm{const}\\
        \partial_x  \alpha_i = 0 \Longrightarrow & \quad \alpha_i = \textrm{const}, \textrm{ for } i=2,\ldots, N,
    \end{eqnarray}
    and according to the same derivations as in \cite{Koellermeier2020i} Equations (4.29)-(4.34), we obtain the result of this theorem.
\end{proof}
As usual, the Froude number distinguishes sub-critical and super-critical flow based on $\Fr<1$ and $\Fr>1$, respectively. The first moment number $\Ma_1$ was already introduced in \cite{Koellermeier2020i} and quantifies deviations from constant velocity, with typically $\Ma_1 < 1$.

We note that (same as the eigenvalues) the steady states of the PHSWME depend only on the first moment $\alpha_1$.

\begin{remark}
    We note that it is not clear a-priori that any regularization is effective at obtaining a PDE system with the desired properties. This is also the case in the earlier hyperbolic regularizations in kinetic theory, see \cite{Cai2013b,Koellermeier2015}. In fact, for the rarefied gas dynamics models in \cite{Cai2013b}, the regularization was using convective variables as a regularization in primitive variables would have resulted in more changes with respect to the original model.
    In our case of shallow flows, the regularization in primitive variables is equally intuitive (or non-intuitive) as the regularization in convective variables as the full SWME model does already contain non-conservative products from the start. This is different from the models in kinetic theory in \cite{Cai2013b}.

    In future work it would be interesting to investigate whether any of the current regularizations could be linked to a regularization or closure of the velocity profile itself. Some inspiration might be taken from the related moment models in gas dynamics \cite{McDonald2013,Schaerer2015} where a closure relation determines the dependencies of the moments based on some underlying assumptions.
\end{remark}

\section{Primitive Moment regularization for Hyperbolic SWME (PMHSWME)}
\label{sec:new_new_hyp}
The previous section showed that PHSWME, a model that regularizes the primitive system matrix, succeeds at both restoring hyperbolicity as well as having analytical steady states.
However, both the eigenvalues as well as the steady states of the PHSWME depend only on the first coefficient $\alpha_1$ and not on the higher-order coefficients $\alpha_i$ for $i\geq2$, which can be seen as a simplification of physics.

We remember the SWLME from \eqref{eq:A_SWLME_c} and \cite{Koellermeier2020i}, which neglects most nonlinear effects in the transport part of the higher-order moment equations while keeping the mass and momentum equation for $h$ and $h u_m$ unchanged. Despite the resulting simplification of the system matrix, this resulted in physically interpretable, analytical steady states \cite{Koellermeier2020i}, depending on all involved higher-order moments. The same was the case for the eigenvalues, which then depend also on the higher-order coefficients, see \eqref{eq:SWLME_c_EV}.

We already saw with the MHSWME in theorem \ref{th:hybrid_convective_system} that keeping the momentum equation unchanged and regularizing only the moment equations in convective variables only leads to a locally hyperbolic system with infeasible steady states.

We now combine the ideas of PHSWME and MHSWME and derive a new regularized system matrix overcoming both problems. As in the case of PHSWME, we start from the formulation in primitive variables, see \eqref{eq:A_SWME_p}. This time, we keep the first two equations unchanged (like in the SWLME and MHSWME) and only regularize the remaining moment equations $2+i$, for $i=1, \ldots, N$ around linear velocity profiles. This means that all coefficients $\alpha_i$ with $i\geq2$ are set to zero in the last $N$ equations only. We thus denote the resulting system as the Primitive Moment regularization for Hyperbolic SWME (PMHSWME), with the explicit equations given by the next theorem.

\begin{theorem}\label{th:APMHSWME_p}
    Regularizing only the last $N$ rows of the primitive system matrix \eqref{eq:A_SWME_p} around linear velocity profiles, i.e., evaluating at $U_{p}=\left(h,u_m,\alpha_1, 0, \ldots, 0\right)$, the Primitive Moment regularization for Hyperbolic Shallow Water Moment Equations (PHSWME) is obtained as
    \begin{equation}\label{eq:PMHSWME_p}
        \frac{\partial U_{p}}{\partial t} + A^{\textrm{PMHSWME}}_{p} \frac{\partial U_{p}}{\partial x} = 0,
    \end{equation}
    using the primitive system matrix $A^{\textrm{PMHSWME}}_{p} \in \mathbb{R}^{(N+2)\times(N+2)}$ defined as
    \begin{equation}\label{eq:A_PMHSWME_p}
        A^{\textrm{PMHSWME}}_{p} =
        \left(\begin{array}{cccccc}
        u_m & h &  &  &  &  \\
        g + \frac{1}{h} \sum\limits_{i=1}^N \frac{\alpha_i^2}{2i+1} & u_m & \frac{2}{3} \alpha_1 & \frac{2}{5} \alpha_2 & \dots & \frac{2}{2N+1} \alpha_N \\
         & \alpha_1 & u_m & \frac{3}{5} \alpha_1 & &\\
        -\frac{\alpha_1^2}{3h} &   & \frac{1}{3}\alpha_1 & u_m & \ddots & \\
         &  &  & \ddots & \ddots & \frac{N+1}{2N+1}\alpha_1 \\
         &  &  &  & \frac{N-1}{2N-1}\alpha_1 & u_m
        \end{array}\right).
    \end{equation}
\end{theorem}

\begin{proof}
    The proof is trivial and follows from Lemma \ref{lemma:SWME_p} and Lemma \ref{lemma:APHSWME_p} by not modifying the second equation.
\end{proof}
We note again that by construction only the second equation differs when comparing the two regularized models PHSWME \eqref{eq:PHSWME_p} and PMHSWME \eqref{eq:PMHSWME_p} in primitive variables.

To show hyperbolicity, we use the PMHSWME in primitive variables. Similar to the PHSWME, the following theorem computes the characteristic polynomial and  eigenvalues of the PMHSWME system matrix $A^{\textrm{PMHSWME}}_{p}$ from \eqref{eq:A_PMHSWME_p}.

\begin{theorem}\label{th:APMHSWME_p_chi}
    The matrix $A^{\textrm{PMHSWME}}_{p} \in \mathbb{R}^{(N+2)\times(N+2)}$ defined in Lemma \ref{th:APMHSWME_p} has the following characteristic polynomial
    \begin{equation}\label{eq:APMHSWME_p_char_poly}
        \chi_{A^{\textrm{PMHSWME}}_{p}}(\lambda) = \frac{(-\alpha_1)^NN!}{(2N+1)!!} \cdot P^{'}_{N+1}\left(\frac{\lambda-u_m}{\alpha_1}\right) \cdot \left( \left( \lambda - u_m \right)^2 - gh - \alpha_1^2 - \sum\limits_{i=2}^N\frac{\alpha_i^2}{2i+1}\right),
    \end{equation}
    and the real and pairwise distinct eigenvalues $\lambda_1, \ldots, \lambda_{N+2}$ of $A^{\textrm{PMHSWME}}_{p}$ are given by
    \begin{eqnarray}\label{eq:APMHSWME_p_EV}
        \lambda_i, i=1,\ldots, N: P'_{N+1}\left(\frac{\lambda_i - u_m}{\alpha_1}\right)=0,
        \lambda_{N+1,N+2} = u_m \pm \sqrt{gh + \alpha_1^2 + \sum\limits_{i=2}^N\frac{\alpha_i^2}{2N+1}}.
    \end{eqnarray}
\end{theorem}
\begin{proof}
    For the proof, we proceed similarly as in Theorem \ref{th:APHSWME_p_chi}. Starting with
    \begin{equation}
       A^{\textrm{PMHSWME}}_{p} = \widetilde{A}^{{\textrm{PMHSWME}}} + u_m I_{N+2}, \quad \widetilde{\lambda} = \lambda - u_m,
    \end{equation}
    we will compute
    \begin{equation}
       \chi_{A^{\textrm{PMHSWME}}_{p}}(\lambda) = \det\left( \widetilde{A}^{{\textrm{PMHSWME}}} - \widetilde{\lambda} I_{N+2} \right) =: \left| \widetilde{A}^{{\textrm{PMHSWME}}} - \widetilde{\lambda} I_{N+2} \right|,
    \end{equation}
    with $\widetilde{A}^{{\textrm{PMHSWME}}} =$
    \begin{equation}
        \left(\begin{array}{cccccc}
        - \widetilde{\lambda} & h &  &  &  &  \\
        g + \frac{1}{h} \sum\limits_{i=1}^N \frac{\alpha_i^2}{2i+1} & - \widetilde{\lambda} & \frac{2}{3} \alpha_1 & \frac{2}{5} \alpha_2 & \dots & \frac{2}{2N+1} \alpha_N \\
         & \alpha_1 & - \widetilde{\lambda} & \frac{3}{5} \alpha_1 & &\\
        -\frac{\alpha_1^2}{3h} &   & \frac{1}{3}\alpha_1 & - \widetilde{\lambda} & \ddots & \\
         &  &  & \ddots & \ddots & \frac{N+1}{2N+1}\alpha_1 \\
         &  &  &  & \frac{N-1}{2N-1}\alpha_1 & - \widetilde{\lambda}
        \end{array}\right) =
        \left(\begin{array}{cccccc}
        - \widetilde{\lambda} & d_0 &  &  &  &  \\
        \widetilde{d}_{1} & - \widetilde{\lambda} & \widetilde{d}_{2} & \widetilde{d}_{3} & \dots & \widetilde{d}_{{N+1}} \\
        0 & d_4 & - \widetilde{\lambda} & c_2 & &\\
        d_5 &   & a_2 & - \widetilde{\lambda} & \ddots & \\
         &  &  & \ddots & \ddots & c_N \\
         &  &  &  & a_N & - \widetilde{\lambda}
        \end{array}\right),
    \end{equation}
    where we made the following definitions
    \begin{equation}
       d_0 = h, \quad d_2 = \frac{2}{3}\alpha_{1}, \quad d_4 = \alpha_1, \quad d_5 = -\frac{\alpha_1^2}{3h},
    \end{equation}
    \begin{equation}
      \widetilde{d}_{1} = g + \frac{1}{h} \sum\limits_{i=1}^N \frac{\alpha_i^2}{2i+1}, \quad \widetilde{d}_{i} = \frac{2}{2i+1}\alpha_{i-1}, \textrm{ for } i=2,\ldots, N+1,
    \end{equation}
    and as usual we denote $a_i = \frac{i-1}{2i-1}\alpha_1$ and $c_i = \frac{i+1}{2i+1}\alpha_1$ for $i=2,\ldots, N$.
    Note how the entries $\widetilde{d}_{i}$ denote entries that are different with respect to the proof of Theorem \ref{th:APHSWME_p_chi}.

    The desired characteristic polynomial $\left| \widetilde{A}^{{\textrm{PMHSWME}}} - \widetilde{\lambda} I_{N+2} \right|$ is computed by developing the determinant with respect to the first row and then developing with respect to the first column.
    \begin{eqnarray*}
        \left| \widetilde{A}^{{\textrm{PMHSWME}}} - \widetilde{\lambda} I_{N+2} \right| &=& \left|
        \begin{array}{cccccc}
    - \widetilde{\lambda} & d_0 &  &  &  &  \\
    \widetilde{d}_{1} & - \widetilde{\lambda} & \widetilde{d}_{2} & \widetilde{d}_{3} & \dots & \widetilde{d}_{{N+1}} \\
    0 & d_4 & - \widetilde{\lambda} & c_2 & &\\
    d_5 &   & a_2 & - \widetilde{\lambda} & \ddots & \\
     &  &  & \ddots & \ddots & c_N \\
     &  &  &  & a_N & - \widetilde{\lambda}
    \end{array}
        \right| \\
      &=& (-\widetilde{\lambda})
      \left|
        \begin{array}{ccccc}
         -\widetilde{\lambda} & \widetilde{d}_{2} & \widetilde{d}_{3} & \dots & \widetilde{d}_{{N+1}} \\
          d_4 & -\widetilde{\lambda} & c_2  &  &   \\
           & a_2 & -\widetilde{\lambda} & \ddots  &   \\
        &  & \ddots & \ddots & c_N \\
         &  &  & a_N & -\widetilde{\lambda}   \\
        \end{array}
        \right|
        -d_0 \left|
        \begin{array}{ccccc}
        \widetilde{d}_{1} & \widetilde{d}_{2} & \widetilde{d}_{3} & \dots & \widetilde{d}_{{N+1}} \\
         0 &  -\widetilde{\lambda} & c_2  &  &   \\
         d_5 & a_2 & -\widetilde{\lambda} & \ddots  &   \\
        &   & \ddots & \ddots & c_N \\
         &   &  & a_N & -\widetilde{\lambda}   \\
        \end{array}
        \right| \\
        &=& (-\widetilde{\lambda}) \left( (-\widetilde{\lambda})  \left|
        \begin{array}{cccc}
           -\widetilde{\lambda} & c_2  &  &   \\
          a_2 & -\widetilde{\lambda} & \ddots  &   \\
       & \ddots & \ddots & c_N \\
         &  & a_N & -\widetilde{\lambda}   \\
        \end{array}
        \right| -d_4  \left|
        \begin{array}{cccccc}
          \widetilde{d}_{2} & \widetilde{d}_{3} & \dots & \dots & \widetilde{d}_{{N+1}} \\
          a_2 &  -\widetilde{\lambda} & c_3  &  & \\
        & a_3 & -\widetilde{\lambda} & \ddots & \\
        & & \ddots & \ddots & c_N \\
         & &   & a_N & -\widetilde{\lambda}   \\
        \end{array}
        \right| \right) \\
        && - d_0 \left( \widetilde{d}_{1}  \left|
        \begin{array}{cccc}
           -\widetilde{\lambda} & c_2  &  &   \\
          a_2 & -\widetilde{\lambda} & \ddots  &   \\
        & \ddots & \ddots & c_N \\
         &  & a_N & -\widetilde{\lambda}   \\
        \end{array}
        \right| + d_5  \left|
        \begin{array}{cccccc}
         \widetilde{d}_{2} & \widetilde{d}_{3} & \dots & \dots & \dots & \widetilde{d}_{{N+1}} \\
          -\widetilde{\lambda} & c_2 & & & &   \\
         & a_3 & -\widetilde{\lambda} & c_4 & &\\
         &   & a_4 & -\widetilde{\lambda} & \ddots   \\
         &   &  & \ddots & \ddots & c_N   \\
         &   & & & a_N & -\widetilde{\lambda}   \\
        \end{array}
        \right|\right) 
    \end{eqnarray*}
    where the last determinant can be written as
    \begin{equation*}
        \left|
        \begin{array}{cccccc}
         \widetilde{d}_{2} & \widetilde{d}_{3} & \hdots & \dots & \dots & \widetilde{d}_{{N+1}} \\
          -\widetilde{\lambda} & c_2 & & & &   \\
         & a_3 & -\widetilde{\lambda} & c_4 & &\\
         &   & a_4 & -\widetilde{\lambda} & \ddots   \\
         &   &  & \ddots & \ddots & c_N   \\
         &   & & & a_N & -\widetilde{\lambda}   \\
        \end{array}
        \right| = \widetilde{d}_{2} c_2 \left|
        \begin{array}{cccc}
           -\widetilde{\lambda} & c_4  &  &   \\
          a_4 & -\widetilde{\lambda} & \ddots  &   \\
        & \ddots & \ddots & c_N \\
         &  & a_N & -\widetilde{\lambda}   \\
        \end{array}
        \right| - (-\widetilde{\lambda})  \left|
        \begin{array}{cccccc}
          \widetilde{d}_{2} & \widetilde{d}_{3} & \dots & \dots & \widetilde{d}_{{N+1}} \\
          a_2 &  -\widetilde{\lambda} & c_3  &  & \\
        & a_3 & -\widetilde{\lambda} & \ddots & \\
        & & \ddots & \ddots & c_N \\
         & &   & a_N & -\widetilde{\lambda}   \\
        \end{array}
        \right|.
    \end{equation*}
    We reuse the definitions of $|A_i|$ from \eqref{eq:subdeterminant}, and $|D_i|$ from \eqref{subdeterminant_D}, including the recursion formula
    \begin{equation*}
        \left|D_{i}\right| = \frac{1}{a_{i-1}}\left( \widetilde{d}_{{i-1}} \left|A_i\right| - \left|D_{i-1}\right| \right),
    \end{equation*}
    so that we can eliminate both $\left|A_4\right|$ and $\left|D_3\right|$ from the expression and factorize the remaining terms with respect to $\left|A_2\right|$, $\left|A_3\right|$, and $\left|D_2\right|$ as follows
    \begin{eqnarray*}
       \chi_{A^{\textrm{PMHSWME}}_{p}}(\lambda) &=&(-\widetilde{\lambda}) \left((-\widetilde{\lambda}) \left|A_2\right|-d_4\left|D_2\right|\right)-d_0  \left(\widetilde{d}_{1}\left|A_2\right|+d_5 \left(\widetilde{d}_2 c_2\left|A_4\right|-(-\widetilde{\lambda}) \left|D_3\right|\right)\right) \\
       &=& \widetilde{\lambda}^2\left|A_2\right|+\widetilde{\lambda} d_4\left|D_2\right|-d_0 \widetilde{d}_{1}\left|A_2\right|-d_0 d_5 d_2 c_2\left|A_4\right|-\widetilde{\lambda} d_0 \cdot d_5 \cdot\left|D_3\right| \\
       &=& \widetilde{\lambda}^2\left|A_2\right|+\widetilde{\lambda} d_4\left|D_2\right|-d_0 \widetilde{d}_{1}\left|A_2\right|-\frac{d_0 d_5 \widetilde{d}_2 c_2}{\left(-a_2 c_2\right)} \cdot\left(\left|A_2\right|+\lambda^2\left|A_3\right|\right)-\widetilde{\lambda} \frac{d_0 d_5}{a_2}\left(\widetilde{d}_2\left|A_3\right|-\left|D_2\right|\right) \\
        &=& \left|A_2\right| \cdot\left(\widetilde{\lambda}^2-d_0 \widetilde{d}_{1}+\frac{d_0 d_5\widetilde{d}_2}{a_2}\right) + |A_3| \cdot \widetilde{\lambda} \left(0\right) +\left|D_2\right| \cdot  \widetilde{\lambda}\left(d_4+\frac{d_0 d_5}{2}\right) \\
        &=& \left|A_2\right| \cdot\left(\widetilde{\lambda}^2 - gh - \frac{2}{3}\alpha_1^2 - \sum\limits_{i=1}^N\frac{\alpha_i^2}{2i+1}\right) +\left|D_2\right| \cdot  \widetilde{\lambda}\left( 0 \right),
    \end{eqnarray*}
    where we inserted the specific terms for the variables $d_i, \widetilde{d}_i$, and $a_2$. Note how the terms multiplied with $|A_3|$ and $|D_2|$ evaluate to zero and only the known $|A_2|$ from Lemma \ref{lemma:A_2_chi} term remains.
    With the notation $\widetilde{\lambda} = \lambda - u_m$ and $|A_2| = \chi_{A_2}(\widetilde{\lambda})$ we obtain the result.
    \begin{equation*}
        \chi_{A^{\textrm{PMHSWME}}_{p}}(\lambda) = \frac{(-\alpha_1)^NN!}{(2N+1)!!} \cdot P^{'}_{N+1}\left(\frac{\lambda-u_m}{\alpha_1}\right) \cdot \left( \left( \lambda - u_m \right)^2 - gh - \alpha_1^2 - \sum\limits_{i=2}^N\frac{\alpha_i^2}{2i+1}\right).
    \end{equation*}
    The characteristic polynomial directly reveals the eigenvalues, which completes the proof.
\end{proof}

Note how the characteristic polynomial \eqref{eq:APMHSWME_p_char_poly} has a very similar structure as the one in \eqref{eq:APHSWME_p_char_poly}, expect for the additional dependence on the higher-order coefficients $\alpha_i$ for $i\geq2$. This then leads to the similarity of the eigenvalues with \eqref{eq:APHSWME_p_EV} as well.

As before, we now give the system matrix for the PMHSWME system in convective variables.
\begin{lemma}\label{lemma:APMHSWME_c}
    The PMHSWME system defined in \eqref{eq:A_PMHSWME_p} can be transformed from primitive to convective variables to read
    \begin{equation}\label{eq:PMHSWME_c}
        \frac{\partial U_{c}}{\partial t} + A^{\textrm{PMHSWME}}_{c} \frac{\partial U_{c}}{\partial x} = 0,
    \end{equation}
    using the convective PMHSWME system matrix $A^{\textrm{PMHSWME}}_{c} \in \mathbb{R}^{(N+2)\times(N+2)}$ defined as $A^{\textrm{PMHSWME}}_{c}=$
    \begin{equation}\label{eq:A_PMHSWME_c}
        \left(\begin{array}{cccccccc}
          & 1 &  &  &  &  &  &  \\
        -u_m^2 + gh - \sum\limits_{i=1}^N \frac{\alpha_i^2}{2i+1} & 2u_m & \frac{2}{3} \alpha_1 & \frac{2}{5} \alpha_2 & \dots & \frac{2}{2i+1} \alpha_i & \dots & \frac{2}{2N+1} \alpha_N \\
        -2u_m - \frac{3}{5}\alpha_1\alpha_2 & 2\alpha_1 & u_m & \frac{3}{5} \alpha_1 & &&  &  \\
        -u_m \alpha_2 - \frac{4}{7}\alpha_1\alpha_3 - \frac{2}{3}\alpha_1^2 & \alpha_2 & \frac{1}{3}\alpha_1 & u_m & \ddots & &  &  \\
        \vdots & \vdots &  & \ddots & \ddots & \frac{i+1}{2i+1}\alpha_1 &  &   \\
        -u_m \alpha_i- \frac{i-1}{2i-1}\alpha_1\alpha_{i-1} - \frac{i+1}{2i+1}\alpha_i \alpha_{i+1}&  \alpha_i &  &  & \frac{i-1}{2i-1}\alpha_1 & u_m & \ddots &  \\
        \vdots & \vdots &  &  &  & \ddots & \ddots &  \frac{N+1}{2N+1}\alpha_1 \\
        -u_m \alpha_N - \frac{N-1}{2N-1} \alpha_1 \alpha_N & \alpha_N &  &  &  &  &  \frac{N-1}{2N-1}\alpha_1 & u_m
        \end{array}\right).
    \end{equation}
\end{lemma}
\begin{proof}
    We note that only the second row of the system matrix in \eqref{eq:A_PMHSWME_c} changes with respect to the matrix from \eqref{eq:A_PHSWME_c}.
    The proof is analogous to the proof of Theorem \ref{lemma:APHSWME_c} and is left out here for brevity.
\end{proof}

Note how the only differences with respect to the PHSWME from Theorem \ref{lemma:APHSWME_c} are in the second row of the matrix. The second row is now the same as for the full SWME system \eqref{eq:A_SWME_c}, the SWLME derived in \cite{Koellermeier2020i}, and the MHSWME \eqref{eq:A_MHSWME_c}. In other words, the momentum balance equation of the PMHSWME is unchanged with respect to the SWME.

Using Theorem \ref{th:APHSWME_steady_states} and the similarity with the SWLME, we can derive analytical steady states of the PMHSWME system.

\begin{theorem}\label{th:APMHSWME_steady_states}
    The PMHSWME system defined in \eqref{eq:A_PMHSWME_p} allows for the following steady states
    \begin{equation}\label{eq:APMHSWME_steady_states}
        h=h_0 \quad \textrm{or} \quad -\Fr^2 + \frac{1}{2} \left(\left(\frac{h}{h_0}\right)^2+\left(\frac{h}{h_0}\right)\right) + \sum\limits_{i=1}^{N}\frac{\Ma_i^2 \Fr^2}{2i+1}  \left( \left(\frac{h}{h_0}\right)^3+\left(\frac{h}{h_0}\right)^2+\left(\frac{h}{h_0}\right)\right)=0,
    \end{equation}
    where, in addition to the Froude number $\Fr=\frac{u_m}{\sqrt{gh}}$, $\Ma_i = \frac{\alpha_{i,0}}{u_{m,0}}$ denotes the dimensionless $i$-th moment number for $i=1, \ldots, N$ at a reference state.
\end{theorem}
\begin{proof}
    We first use the equality of the first equation and the last $N$ equations with the PHSWME system defined in \eqref{eq:A_PHSWME_c}. This lets us derive the steady state conditions
    \begin{eqnarray}
        \partial_x\left( h u_m \right) &=& 0\\
        \partial_x \left( \frac{\alpha_1}{h} \right) = 0 &\textrm{and}& \partial_x  \alpha_i = 0  \textrm{ for } i=2,\ldots, N.
    \end{eqnarray}
    As the second equation is the same as the second equation in the SWLME from \cite{Koellermeier2020i}, we can analogously derive that all steady states fulfill
    \begin{equation}
        \partial_x \left( \frac{1}{2}gh^2 + h u_m^2 + \sum\limits_{i=1}^{N}\frac{1}{2i+1} \alpha_i^2 \right) = 0.
    \end{equation}
    This directly leads to the condition for the steady states in analogy to Theorem \ref{th:APHSWME_steady_states}.
\end{proof}
The steady states thus also depend on the higher-order coefficients. Same as in \cite{Koellermeier2020i}, a new dimensionless number $\Ma^2:=\sum\limits_{i=1}^N \frac{1}{2 i+1}(\Ma)_i^2$ appears. $\Ma$ can be interpreted as measuring a weighted deviation from equilibrium, including all coefficients weighted by the norm of the Legendre polynomials.\\

\subsection{Model property comparison}
Before the numerical comparison, a short summative comparison of the proven analytical properties of all models is in order. We have visualized this concisely in Table \ref{table:models}. It becomes clear that the previously existing models SWME, HSWME, SWLME have individual disadvantages, that cannot be overcome by the new combination MHSWME. However, via the primitive regularization and the PHSWME, we were able to derive the PMHSWME, which finally fulfills all desired properties.
\begin{table}
\begin{center}
\begin{tabular}{ ||P{2.9cm}|P{3.2cm}|P{1.4cm}|P{1.5cm}|P{1.4cm}|P{1.4cm}||}
 \hline \hline
 model acronym & model name & glob. hyperbolic & pres. momentum & analy. steady states & nonlin. moment eqns \\ \hline \hline
 SWE & Shallow Water Equations & $\checkmark$ & $\checkmark$ & $\checkmark$ & ${\color{red} {\color{red} \times}}$ \\ \hline
 SWME \cite{Kowalski2019}: \eqref{eq:SWME_c}, \eqref{eq:SWME_p} & Shallow Water Moment Equations & ${\color{red} \times}$ & $\checkmark$ & ${\color{red} \times}$ & $\checkmark$\\ \hline
 HSWME \cite{Koellermeier2020c}:  \eqref{eq:A_HSWME_c} & Hyperbolic SWME & $\checkmark$ & ${\color{red} \times}$ & ${\color{red} \times}$ & $\checkmark$ \\ \hline
 SWLME \cite{Koellermeier2020i}: \eqref{eq:A_SWLME_c} & Shallow Water Linearized Moment Equations & $\checkmark$ & $\checkmark$ & $\checkmark$ & ${\color{red} \times}$ \\ \hline
 MHSWME (new): \eqref{eq:A_MHSWME_c} & Moment regularization for HSWME & ${\color{red} \times}$ & $\checkmark$ & ${\color{red} \times}$ & $\checkmark$ \\ \hline
 PHSWME (new): \eqref{eq:A_PHSWME_p}, \eqref{eq:A_PHSWME_c} & Primitive regularization for HSWME & $\checkmark$ & ${\color{red} \times}$ & $\checkmark$ & $\checkmark$ \\ \hline
 PMHSWME (new): \eqref{eq:A_PMHSWME_p}, \eqref{eq:A_PMHSWME_c} & Primitive Moment regularization for HSWME & $\checkmark$ & $\checkmark$ & $\checkmark$ & $\checkmark$ \\
 \hline \hline
\end{tabular}
\caption{\label{table:models} Model properties comparison. All existing models SWME, HSWME, SWLME and even the new MHSWME have different disadvantages. The disadvantages can be successively overcome by the new primitive regularizations PHSWME and PMHSWME.}
\end{center}
\end{table}

\section{Numerical dam-break test case}
\label{sec:numerics}
While the focus of this work is the derivation and analysis of new hyperbolic models for free-surface flows, one numerical test case will be employed to briefly compare the performance of the models regarding accuracy and stability. Further numerical tests are left for future work.

For the time and spatial discretization, the explicit, first-order, path-consistent finite-volume scheme also employed in \cite{Koellermeier2020c} is used.

For simplicity, we employ a Newtonian friction term with bottom slip, also used in \cite{Kowalski2019,Koellermeier2020c} with slip length $\lambda = 0.1$ and viscosity $\nu = 0.1$. Other parameters are possible but we focus on the transport part in this paper and do not vary the friction terms.

We compare the new and existing hyperbolic models with the original SWME \eqref{eq:A_SWME_c} and aim for small deviations, since it means that despite the regularization, the accuracy is kept high, while at the same time the analytical properties of the models are improved, e.g., with respect to hyperbolicity, steady states.

For this dam-break test case, we consider an open domain $x\in [-1,1]$ with $n_x=1000$ equidistant cells and simulate until $t_{\textrm{end}}=0.2$ with $\Delta t = 0.005$. The initial water height is given by the step function
\begin{equation}
    h(0,x) = \begin{cases}
\begin{array}{cc}
 1.5 & x \leq 0, \\
 1 & x > 0, \\
\end{array}
\end{cases}
\end{equation}
with a linear velocity profile
\begin{equation}
    u(0,x,\zeta) = 0.5 \zeta,
\end{equation}
leading to the velocity profile coefficients
\begin{equation}
    u_m(0,x) = 0.25, \quad \alpha_1(0,x) = -0.25, \quad \alpha_i(0,x) = 0, \text{ for } i = 2,\ldots, N,
\end{equation}
for moment models ranging from $N=1,\ldots,5$.
The same test case and velocity profile was also used in \cite{Koellermeier2020i}.

For this dam-break test case, the relative $L_1$-errors with respect to the SWME model are plotted in Figure \ref{fig:error_dambreak}. 
Firstly, all models result in relatively accurate solutions with relative errors below $7\%$. Furthermore, we can conclude that SWLME performs the worst for all flow variables, which might be due to the strong regularization assumptions of this approach. Next, HSWME performs relatively bad for water height $h$, mean velocity $u_m$, and $\alpha_2$. The MHSWME only results in relatively large errors for $\alpha_2$, indicating that the unchanged momentum equation leads to more accurate water height $h$ and average velocity $u_m$. The PHSWME (which includes changes in the momentum equation) performs relatively bad for water height $h$ and average velocity $u_m$, which could be attributed to the deleted terms in the momentum equation and the different momentum-height coupling. The PMHSWME results in the most accurate solutions for all variables (together with PHSWME in the case of $\alpha_1$). It is expected that retaining more terms from the SWME model leads to less deviations from the SWME results. However, the effect of that requires a more detailed study than what can be presented in this paper. We therefore leave this for future work.
This test case demonstrates the beneficial regularization strategy of the PMHSWME, only regularizing the higher-order moment equations and keeping the mass and momentum equation unchanged.

\begin{figure}[h!]
    \centering
    \begin{subfigure}[b]{0.47\textwidth}
        \centering
        \includegraphics[width=\textwidth]{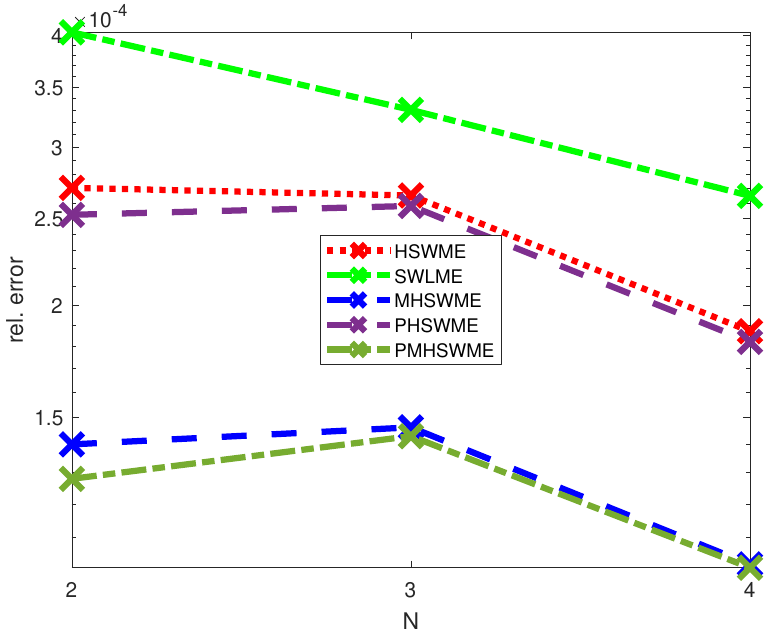}
        \caption{water height $h$}
        \label{fig:error_dambreak_h}
    \end{subfigure}
    \hfill
    \begin{subfigure}[b]{0.47\textwidth}
        \centering
        \includegraphics[width=\textwidth]{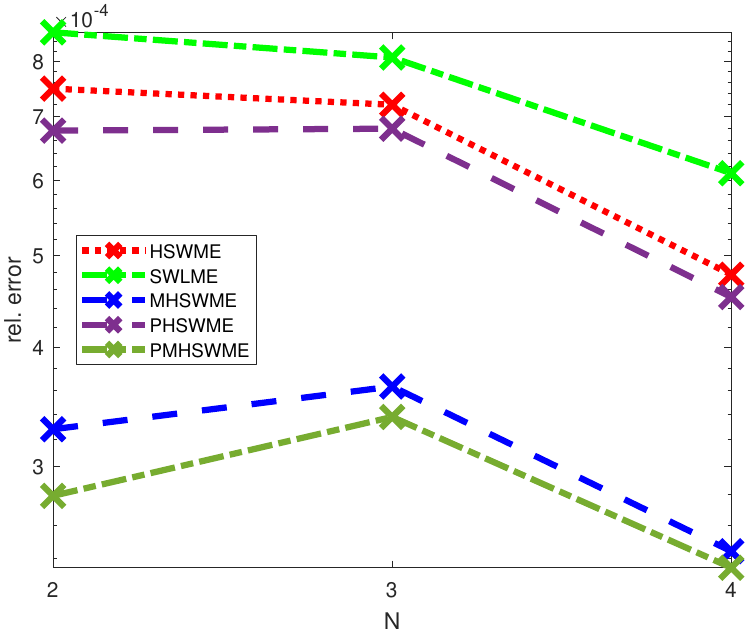}
        \caption{mean velocity $u_m$}
        \label{fig:error_dambreak_u}
    \end{subfigure}
    \vskip\baselineskip
    \begin{subfigure}[b]{0.47\textwidth}
        \centering
        \includegraphics[width=\textwidth]{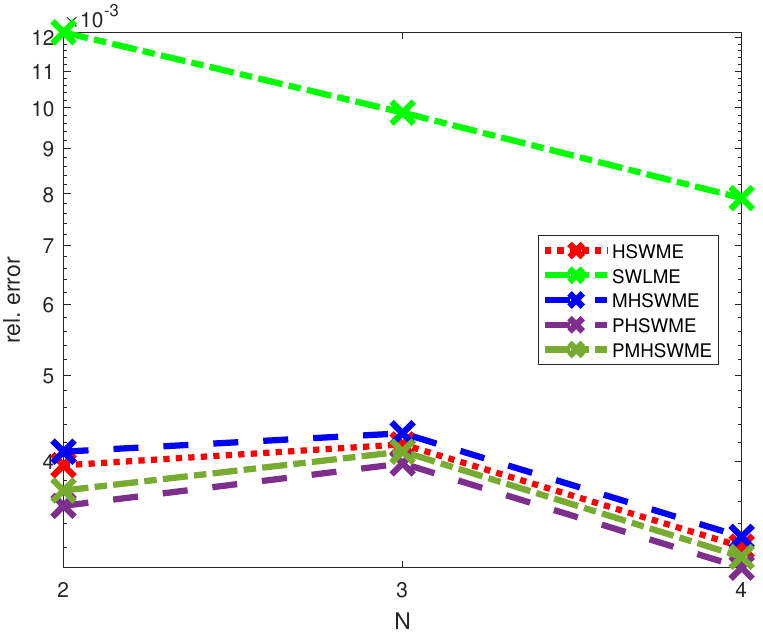}
        \caption{linear coefficient $\alpha_1$}
        \label{fig:error_dambreak_alpha1}
    \end{subfigure}
    \hfill
    \begin{subfigure}[b]{0.47\textwidth}
        \centering
        \includegraphics[width=\textwidth]{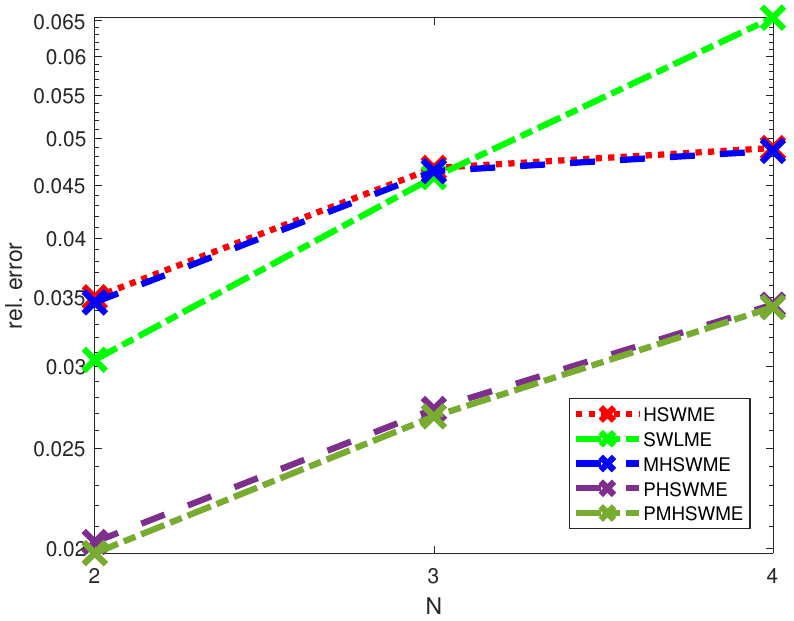}
        \caption{quadratic coefficient $\alpha_2$}
        \label{fig:error_dambreak_alpha2}
    \end{subfigure}
    \caption{Dam break test case relative $L_1$-errors w.r.t. SWME with the same number of moments $N=2,3,4$ for HSWME (red), SWLME (green), MHSWME (blue), PHSWME (purple), PMHSWME (olive). Water height $h$ (top left), mean velocity (top right), linear coefficient $\alpha_1$ (bottom left), and quadratic coefficient $\alpha_2$ (bottom right). Existing models like HSWME and SWLME perform bad for most variables while the new model PMHSWME shows best results for most variables.}
    \label{fig:error_dambreak}
\end{figure}

In addition, we  provide the relative $L_1$-errors with respect to the free-surface solver used as reference in \cite{Kowalski2019} in Figure \ref{fig:error_ref_dambreak}. This reference solution is based on a direct discretization of the transformed reference system derived in \cite{Kowalski2019} on a grid of $n_x=160$ cells in horizontal direction and $n_{\zeta} = 50$ cells in vertical direction. Based on the fact that the SWME model converges to the reference solution as shown in \cite{Kowalski2019} and all our tested models from \ref{fig:error_dambreak} are relatively close to the SWME model, we can infer also convergence of our new models. This is confirmed by the results in Figure \ref{fig:error_ref_dambreak}. In addition to Figure \ref{fig:error_dambreak}, we have also plotted the relative errors of SWME. As expected, the SWME model has the smallest error for larger $N$. Interestingly, the HSWME model seems to be very accurate for the water height. Furthermore, the SWLME model stops converging for the linear coefficient $\alpha_1$. As reported before, the new models PMHSWME and also PHSWME perform among the best models for all variables. This again underlines the benefits of the primitive regularization strategy introduced in this paper.

\begin{figure}[h!]
    \centering
    \begin{subfigure}[b]{0.47\textwidth}
        \centering
        \includegraphics[width=\textwidth]{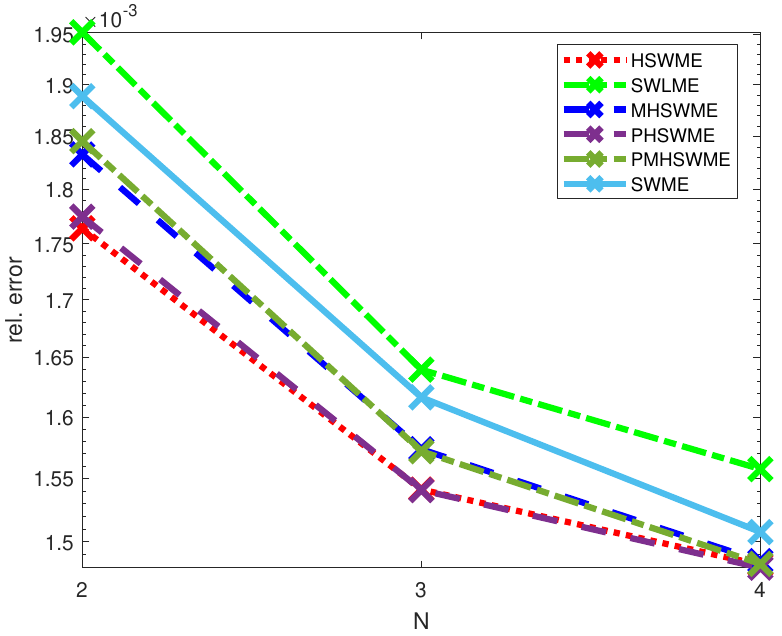}
        \caption{water height $h$}
        \label{fig:error_ref_dambreak_h}
    \end{subfigure}
    \hfill
    \begin{subfigure}[b]{0.47\textwidth}
        \centering
        \includegraphics[width=\textwidth]{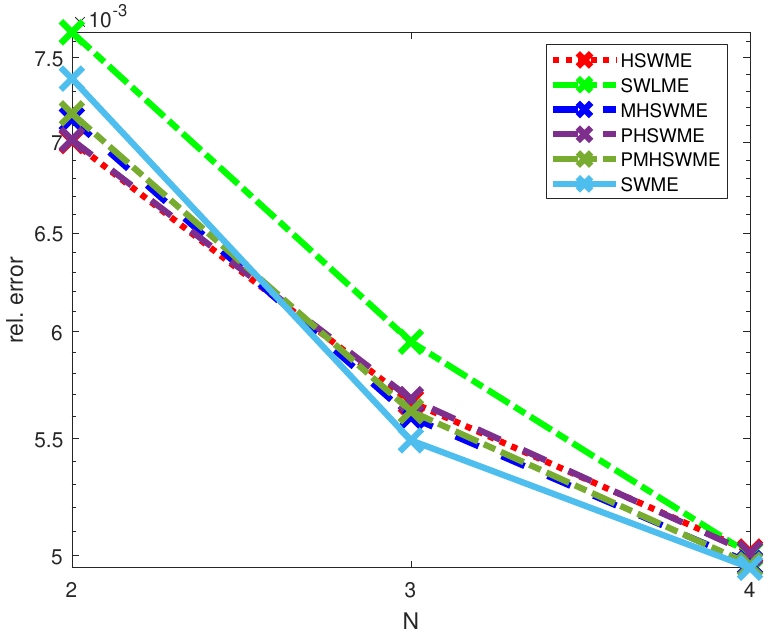}
        \caption{mean velocity $u_m$}
        \label{fig:error_ref_dambreak_u}
    \end{subfigure}
    \vskip\baselineskip
    \begin{subfigure}[b]{0.47\textwidth}
        \centering
        \includegraphics[width=\textwidth]{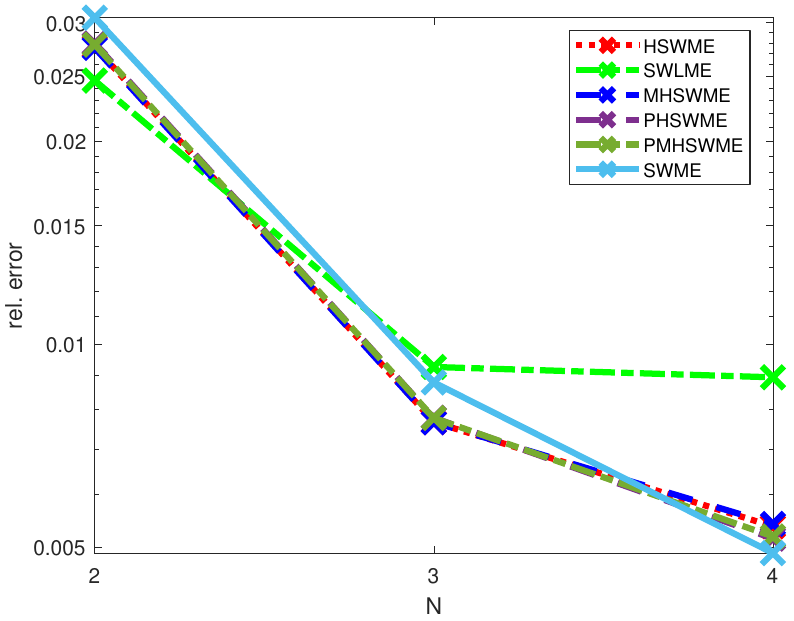}
        \caption{linear coefficient $\alpha_1$}
        \label{fig:error_ref_dambreak_alpha1}
    \end{subfigure}
    \hfill
    \begin{subfigure}[b]{0.47\textwidth}
        \centering
        \includegraphics[width=\textwidth]{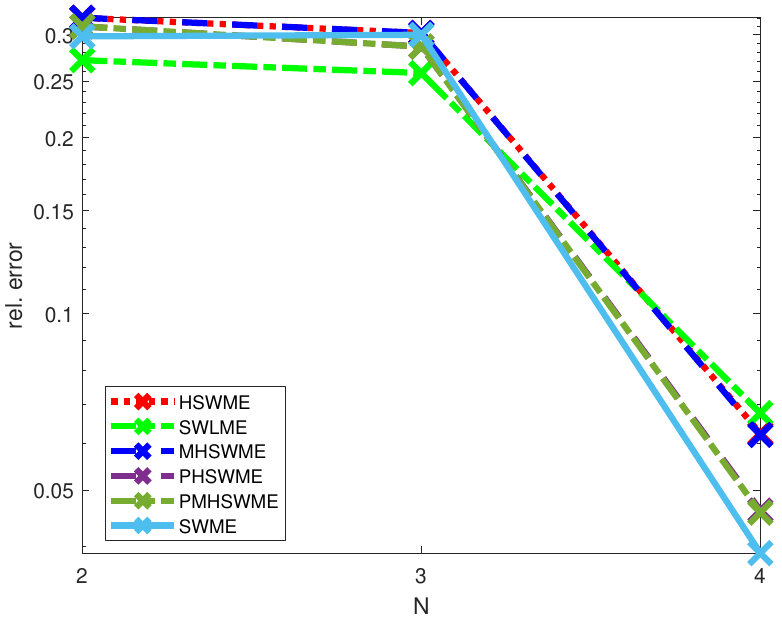}
        \caption{quadratic coefficient $\alpha_2$}
        \label{fig:error_ref_dambreak_alpha2}
    \end{subfigure}
    \caption{Dam break test case relative $L_1$-errors w.r.t. reference solver and $N=2,3,4$ for HSWME (red), SWLME (green), MHSWME (blue), PHSWME (purple), PMHSWME (olive), SWME(light blue). Water height $h$ (top left), mean velocity (top right), linear coefficient $\alpha_1$ (bottom left), and quadratic coefficient $\alpha_2$ (bottom right). Existing models like HSWME and SWLME perform bad for several variables while the new models like PHSWME and PMHSWME show best results for most variables.}
    \label{fig:error_ref_dambreak}
\end{figure}

For more detailed results on the dam-break solutions for $h, u_m, \alpha_1, \alpha_2$, we refer to Figures \ref{fig:sol_dambreak_h}, \ref{fig:sol_dambreak_u}, \ref{fig:sol_dambreak_alpha1}, \ref{fig:sol_dambreak_alpha2}, respectively, below.

More numerical test cases are possible. A strategy to obtain well-balanced numerical solutions for the SWLME model was already developed in \cite{Koellermeier2020i}, based on the framework developed in \cite{Castro2020}. The main ingredient necessary is the analytical computation of steady states, potentially by solving a non-linear system of equations. We also note the existence of other methods like the global flux method, which was applied to shallow water moment models in \cite{CIALLELLA2026106887}.
We leave more detailed numerical test cases and the demonstration of well-balancing properties for the analytically obtained steady states of the new models for future work.

\section{Conclusion and future work}
We derived new Shallow Water Moment Equations, for which global hyperbolicity is proven, analytical steady states are computed, and high accuracy is achieved in a numerical test. This was made possible by transforming the PDE system to primitive variables and performing a hyperbolic regularization in the primitive variable setting. In primitive variables, hyperbolicity could be analyzed efficiently and via a similarity transformation back to convective variables it was clear that the system is hyperbolic. The analytical derivation of steady states was proven to be a further advantage of the new models. A numerical dam-break test case has shown that regularizing only the higher-order moment equations and keeping the momentum equation unchanged increases the accuracy. Importantly, we showed that this only leads to a globally hyperbolic system for the primitive regularization in PMHSWME and not for the convective version MHSWME.

The results of this paper open the possibility for a variety of follow-up projects. Firstly, a more detailed numerical comparison of the different models for different test cases should be performed. Furthermore, the analysis of equilibrium states, see \cite{Huang2022}, energy/entropy properties, see \cite{Gassner2016}, and the adaptation of structure-preserving numerical schemes \cite{Jin2010} can be envisioned. Ultimately, an extension to the multi-dimensional case and inclusion of more physical friction terms for real-world simulations should be the goal.

\section*{Acknowledgments} \label{section: acknowledgments}
This publication is part of the project \textit{HiWAVE} with file number VI.Vidi.233.066 of the \textit{ENW Vidi} research programme, funded by the \textit{Dutch Research Council (NWO)} under the grant \url{https://doi.org/10.61686/CBVAB59929}.

\begin{figure}[h!]
    \centering
    \begin{subfigure}[b]{0.47\textwidth}
        \centering
        \includegraphics[width=\textwidth]{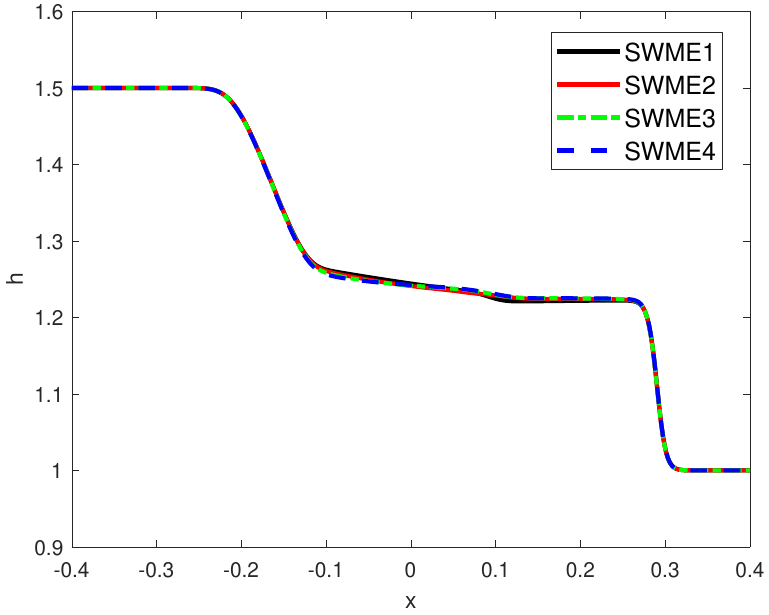}
        \caption{\textrm{SWME}}
        \label{fig:smalldam_SWME_h}
    \end{subfigure}
    \hfill
    \begin{subfigure}[b]{0.47\textwidth}
        \centering
        \includegraphics[width=\textwidth]{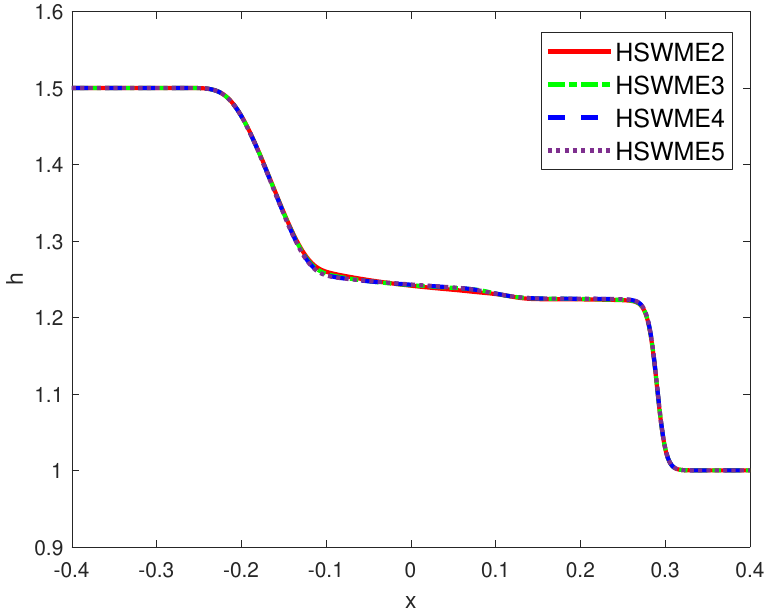}
        \caption{\textrm{HSWME}}
        \label{fig:smalldam_HSWME_h}
    \end{subfigure}
    \vskip\baselineskip
    \begin{subfigure}[b]{0.47\textwidth}
        \centering
        \includegraphics[width=\textwidth]{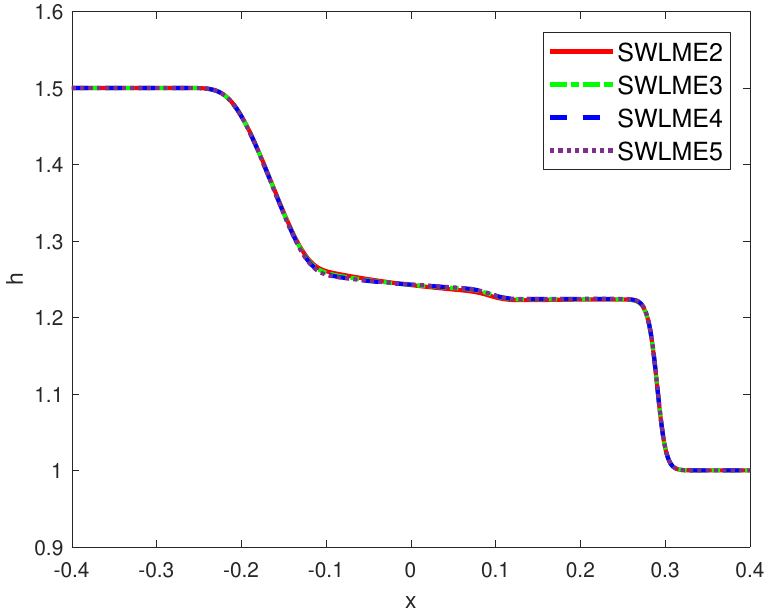}
        \caption{\textrm{SWLME}}
        \label{fig:smalldam_SWLME_h}
    \end{subfigure}
    \hfill
    \begin{subfigure}[b]{0.47\textwidth}
        \centering
        \includegraphics[width=\textwidth]{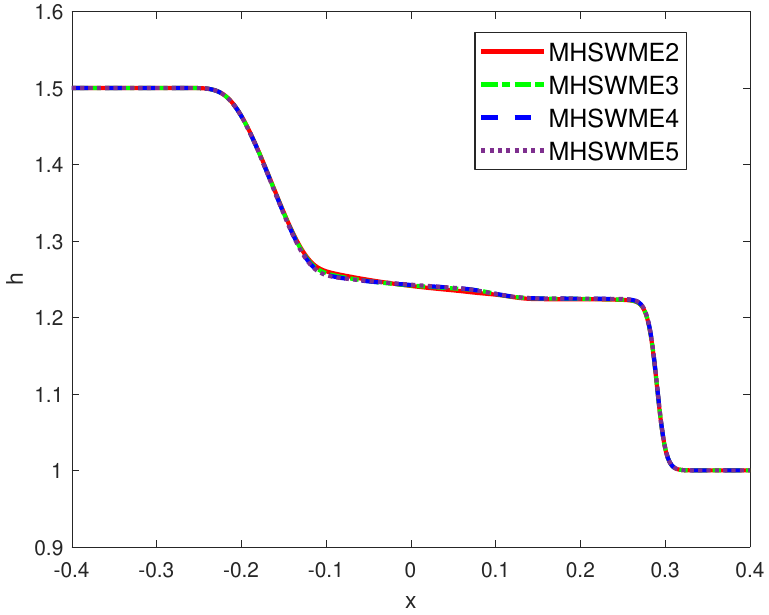}
        \caption{\textrm{MHSWME}}
        \label{fig:smalldam_MHSWME_h}
    \end{subfigure}
    \vskip\baselineskip
    \begin{subfigure}[b]{0.47\textwidth}
        \centering
        \includegraphics[width=\textwidth]{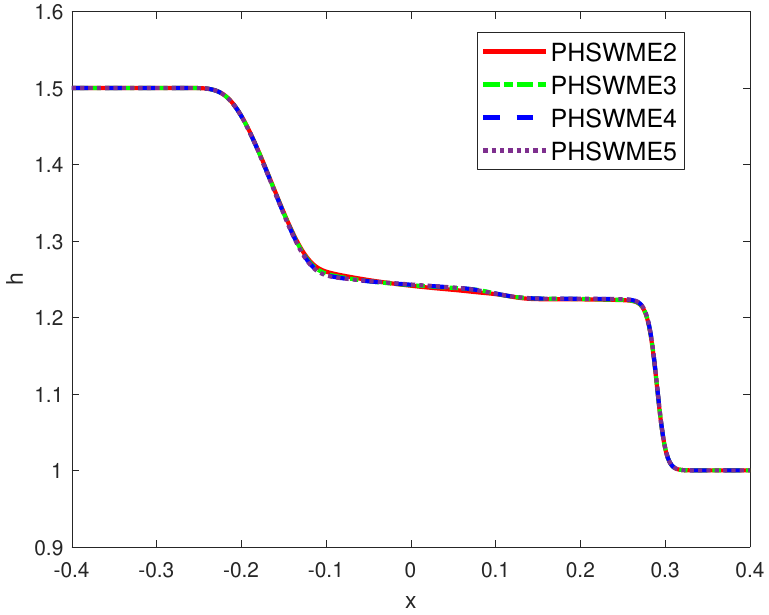}
        \caption{PHSWME}
        \label{fig:smalldam_PHSWME_h}
    \end{subfigure}
    \hfill
    \begin{subfigure}[b]{0.47\textwidth}
        \centering
        \includegraphics[width=\textwidth]{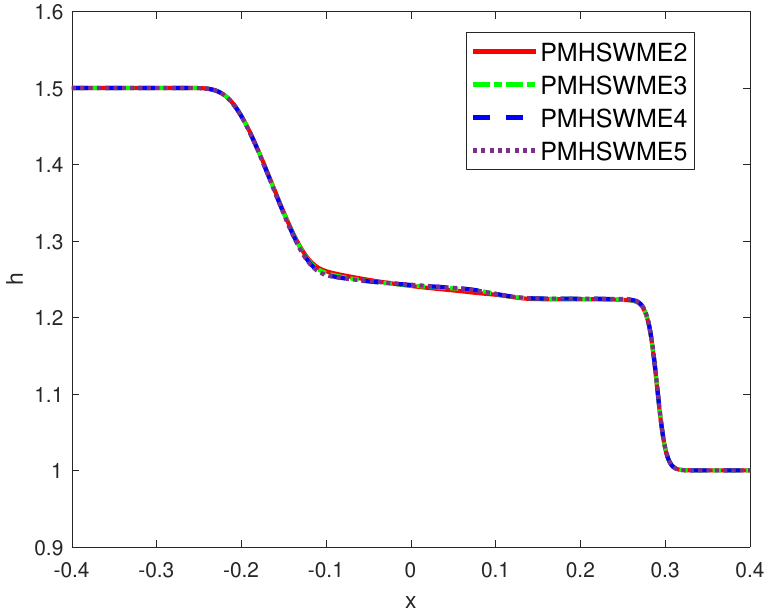}
        \caption{PMHSWME}
        \label{fig:smalldam_PMHSWME_h}
    \end{subfigure}
    \caption{Dam break test case water height $h$ solutions for SWME (top left), HSWME (top right), SWLME (middle left), MHSWME (middle right), PSWME (bottom left), PMSWME (bottom right), for $N=1,2,3,4,5$. Note that all models are equivalent for $N=1$, so this is only shown in the SWME plot. SWME5 is unstable and left out.}
    \label{fig:sol_dambreak_h}
\end{figure}

\begin{figure}[h!]
    \centering
    \begin{subfigure}[b]{0.47\textwidth}
        \centering
        \includegraphics[width=\textwidth]{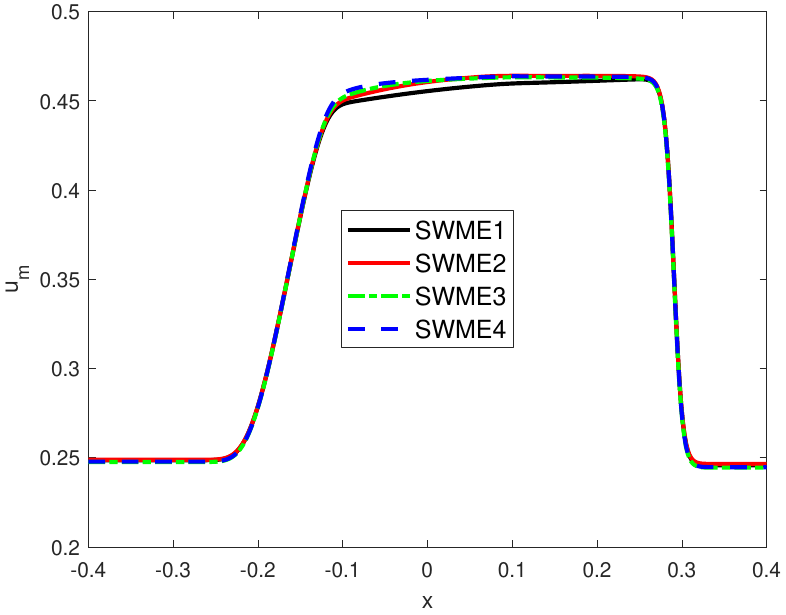}
        \caption{\textrm{SWME}}
        \label{fig:smalldam_SWME_u}
    \end{subfigure}
    \hfill
    \begin{subfigure}[b]{0.47\textwidth}
        \centering
        \includegraphics[width=\textwidth]{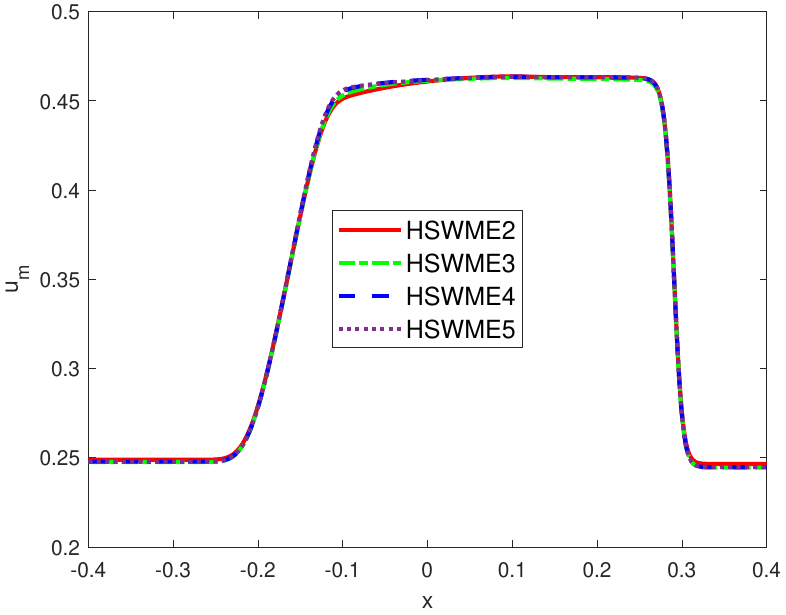}
        \caption{\textrm{HSWME}}
        \label{fig:smalldam_HSWME_u}
    \end{subfigure}
    \vskip\baselineskip
    \begin{subfigure}[b]{0.47\textwidth}
        \centering
        \includegraphics[width=\textwidth]{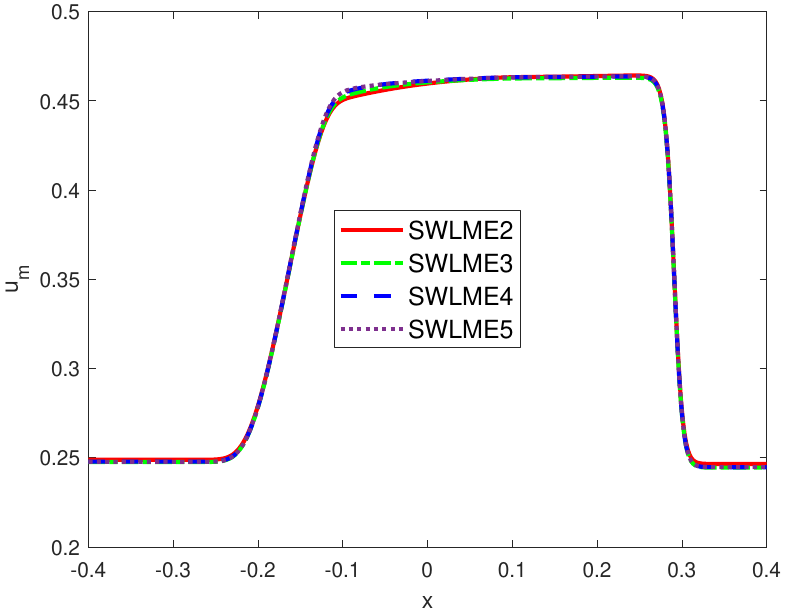}
        \caption{\textrm{SWLME}}
        \label{fig:smalldam_SWLME_u}
    \end{subfigure}
    \hfill
    \begin{subfigure}[b]{0.47\textwidth}
        \centering
        \includegraphics[width=\textwidth]{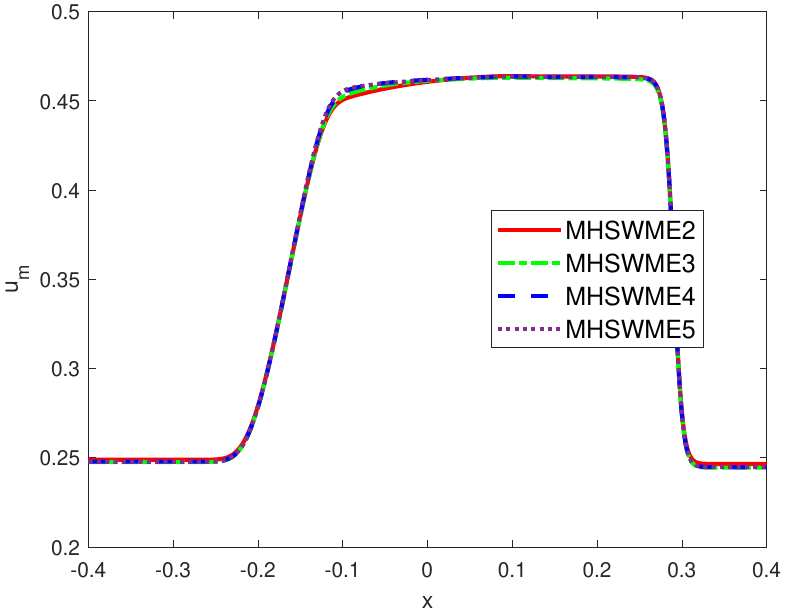}
        \caption{\textrm{MHSWME}}
        \label{fig:smalldam_MHSWME_u}
    \end{subfigure}
    \vskip\baselineskip
    \begin{subfigure}[b]{0.47\textwidth}
        \centering
        \includegraphics[width=\textwidth]{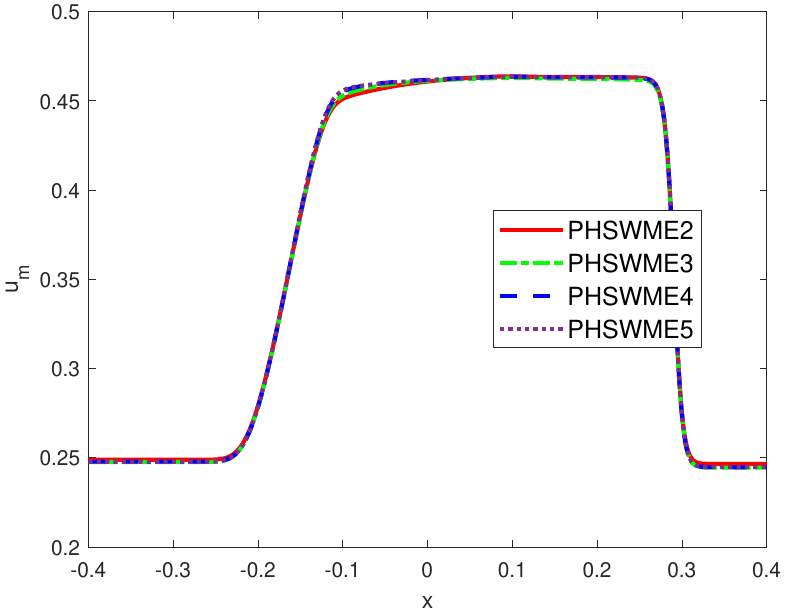}
        \caption{PHSWME}
        \label{fig:smalldam_PHSWME_u}
    \end{subfigure}
    \hfill
    \begin{subfigure}[b]{0.47\textwidth}
        \centering
        \includegraphics[width=\textwidth]{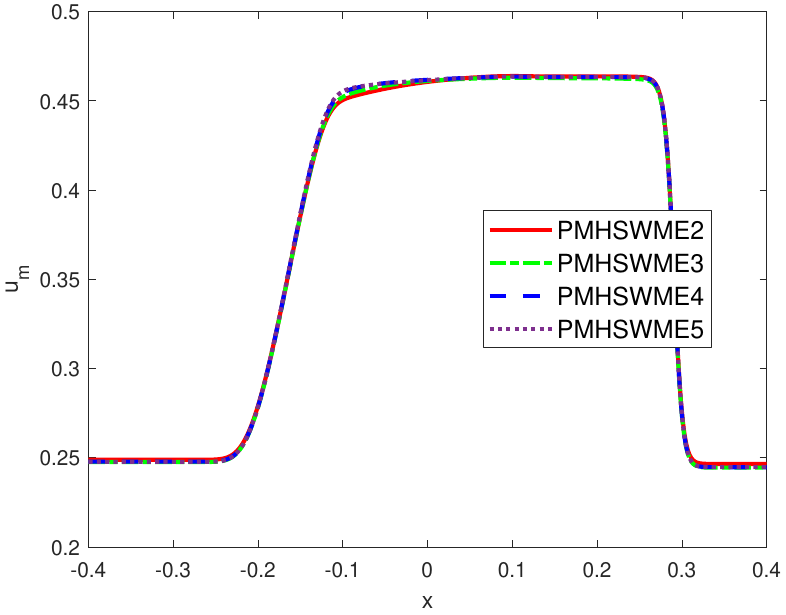}
        \caption{PMHSWME}
        \label{fig:smalldam_PMHSWME_u}
    \end{subfigure}
    \caption{Dam break test case average velocity $u_m$ solutions for SWME (top left), HSWME (top right), SWLME (middle left), MHSWME (middle right), PSWME (bottom left), PMSWME (bottom right), for $N=1,2,3,4,5$. Note that all models are equivalent for $N=1$, so this is only shown in the SWME plot. SWME5 is unstable and left out.}
    \label{fig:sol_dambreak_u}
\end{figure}

\begin{figure}[h!]
    \centering
    \begin{subfigure}[b]{0.47\textwidth}
        \centering
        \includegraphics[width=\textwidth]{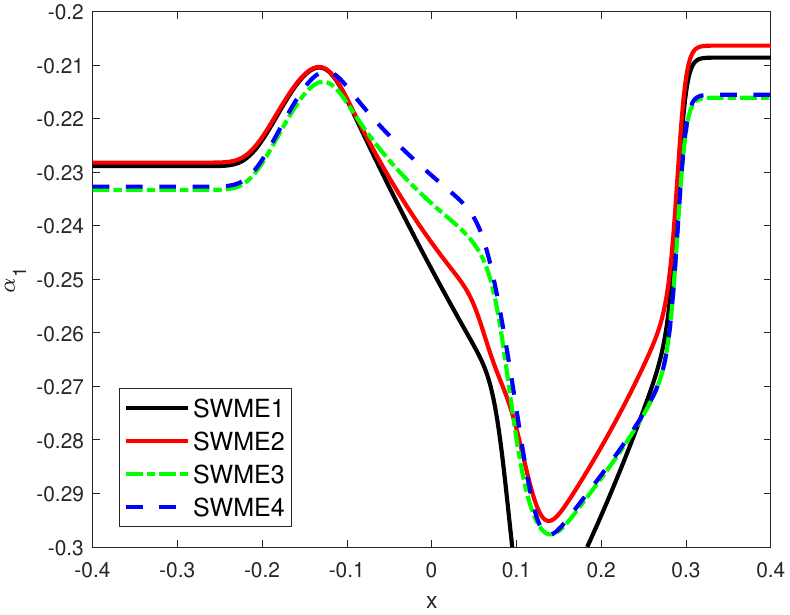}
        \caption{\textrm{SWME}}
        \label{fig:smalldam_SWME_alpha1}
    \end{subfigure}
    \hfill
    \begin{subfigure}[b]{0.47\textwidth}
        \centering
        \includegraphics[width=\textwidth]{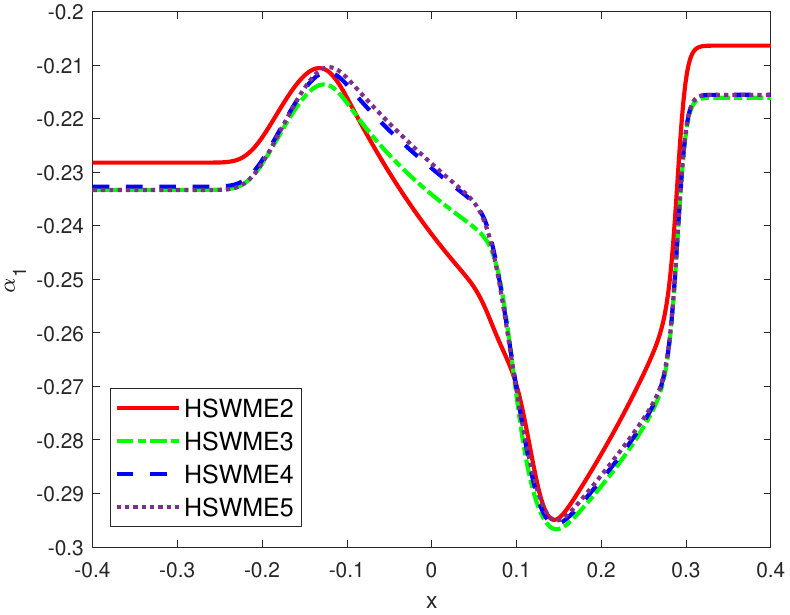}
        \caption{\textrm{HSWME}}
        \label{fig:smalldam_HSWME_alpha1}
    \end{subfigure}
    \vskip\baselineskip
    \begin{subfigure}[b]{0.47\textwidth}
        \centering
        \includegraphics[width=\textwidth]{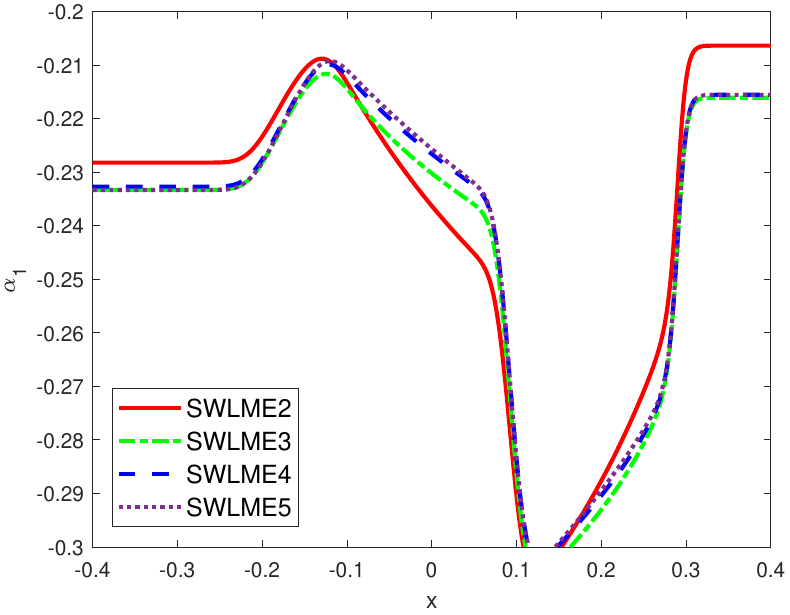}
        \caption{\textrm{SWLME}}
        \label{fig:smalldam_SWLME_alpha1}
    \end{subfigure}
    \hfill
    \begin{subfigure}[b]{0.47\textwidth}
        \centering
        \includegraphics[width=\textwidth]{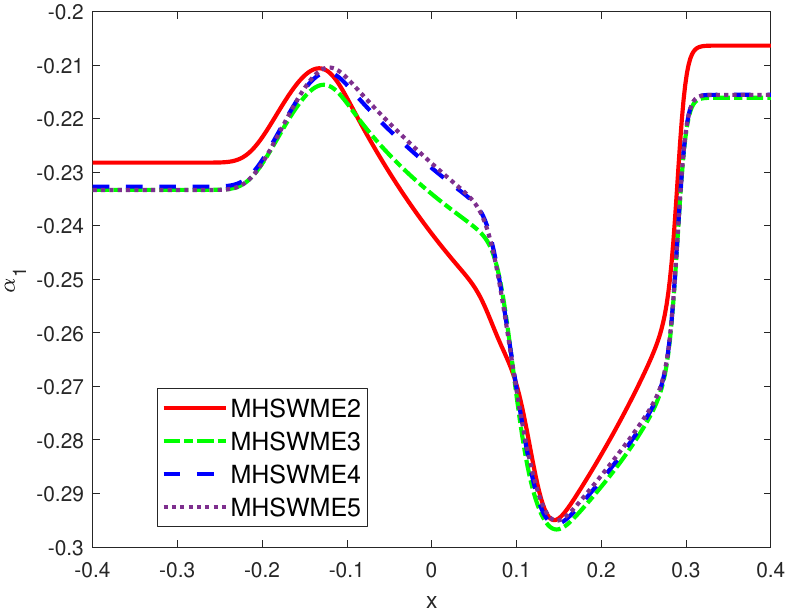}
        \caption{\textrm{MHSWME}}
        \label{fig:smalldam_MHSWME_alpha1}
    \end{subfigure}
    \vskip\baselineskip
    \begin{subfigure}[b]{0.47\textwidth}
        \centering
        \includegraphics[width=\textwidth]{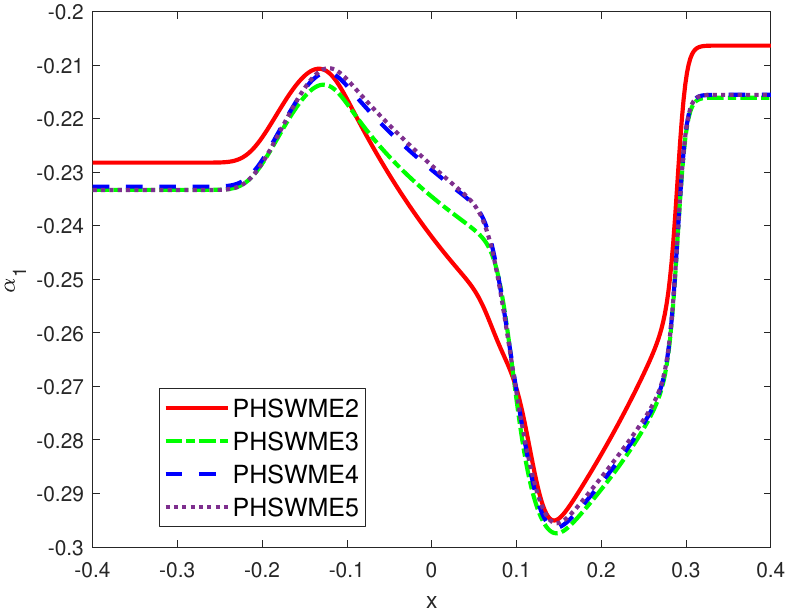}
        \caption{PHSWME}
        \label{fig:smalldam_PHSWME_alpha1}
    \end{subfigure}
    \hfill
    \begin{subfigure}[b]{0.47\textwidth}
        \centering
        \includegraphics[width=\textwidth]{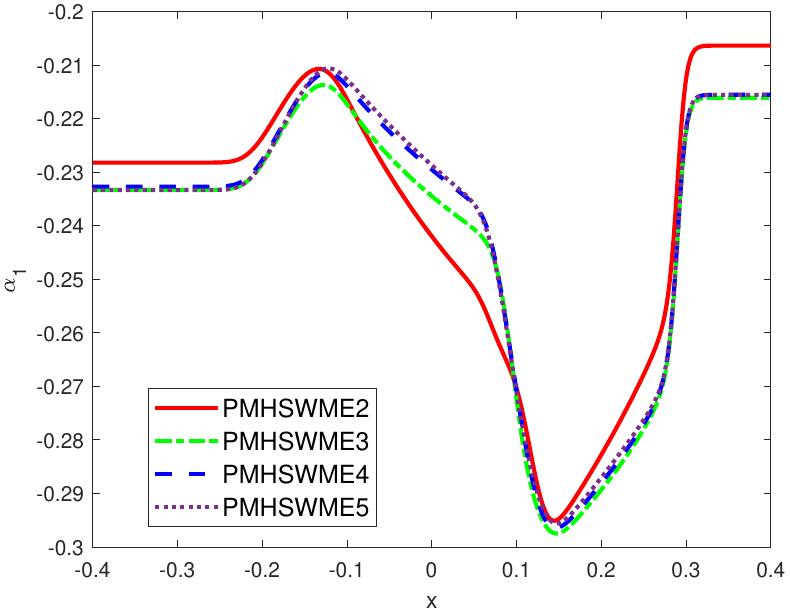}
        \caption{PMHSWME}
        \label{fig:smalldam_PMHSWME_alpha1}
    \end{subfigure}
    \caption{Dam break test case linear coefficient $\alpha_1$ solutions for SWME (top left), HSWME (top right), SWLME (middle left), MHSWME (middle right), PSWME (bottom left), PMSWME (bottom right), for $N=1,2,3,4,5$. Note that all models are equivalent for $N=1$, so this is only shown in the SWME plot. SWME5 is unstable and left out.}
    \label{fig:sol_dambreak_alpha1}
\end{figure}

\begin{figure}[h!]
    \centering
    \begin{subfigure}[b]{0.47\textwidth}
        \centering
        \includegraphics[width=\textwidth]{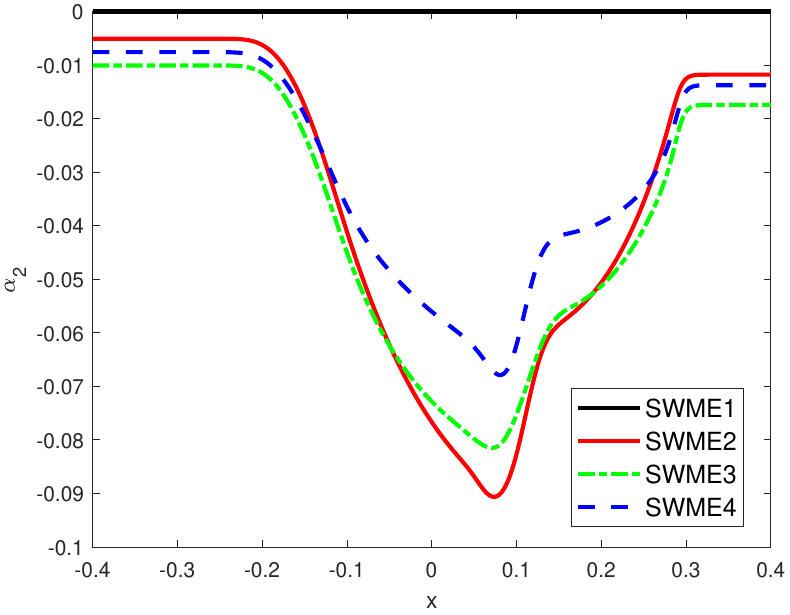}
        \caption{\textrm{SWME}}
        \label{fig:smalldam_SWME_alpha2}
    \end{subfigure}
    \hfill
    \begin{subfigure}[b]{0.47\textwidth}
        \centering
        \includegraphics[width=\textwidth]{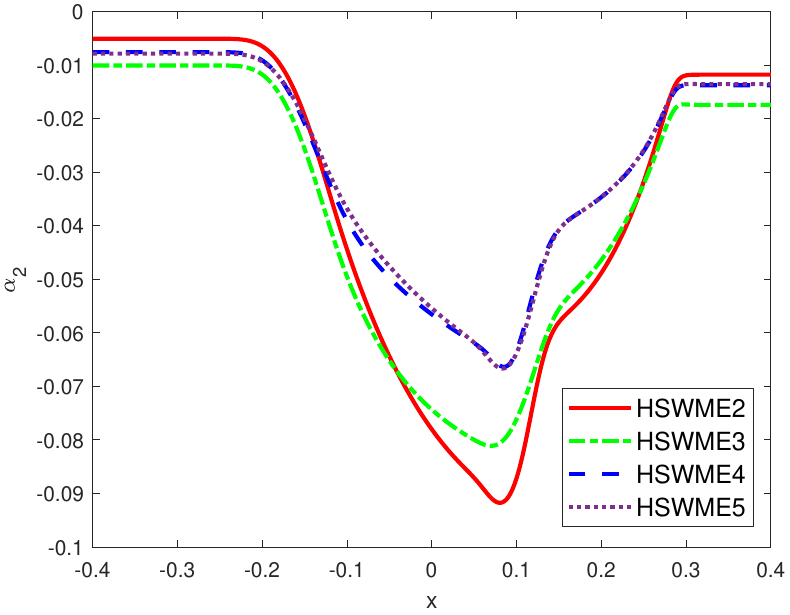}
        \caption{\textrm{HSWME}}
        \label{fig:smalldam_HSWME_alpha2}
    \end{subfigure}
    \vskip\baselineskip
    \begin{subfigure}[b]{0.47\textwidth}
        \centering
        \includegraphics[width=\textwidth]{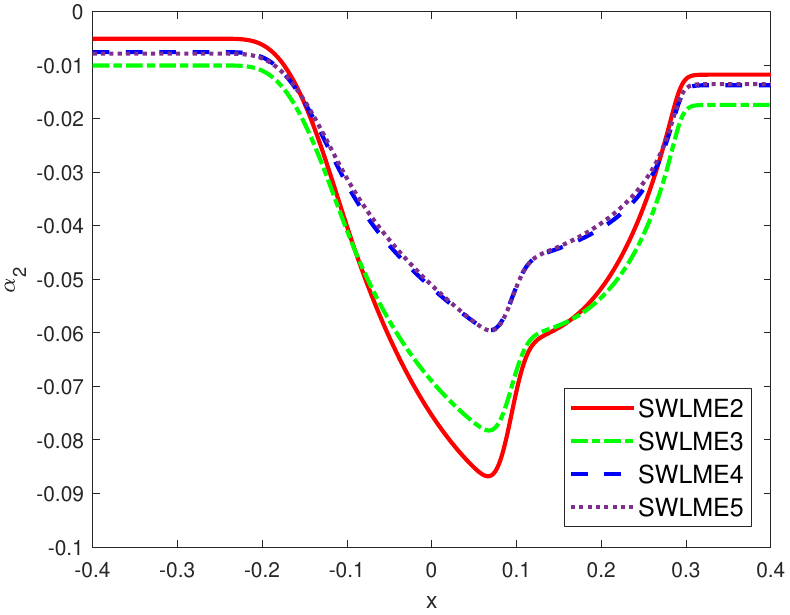}
        \caption{\textrm{SWLME}}
        \label{fig:smalldam_SWLME_alpha2}
    \end{subfigure}
    \hfill
    \begin{subfigure}[b]{0.47\textwidth}
        \centering
        \includegraphics[width=\textwidth]{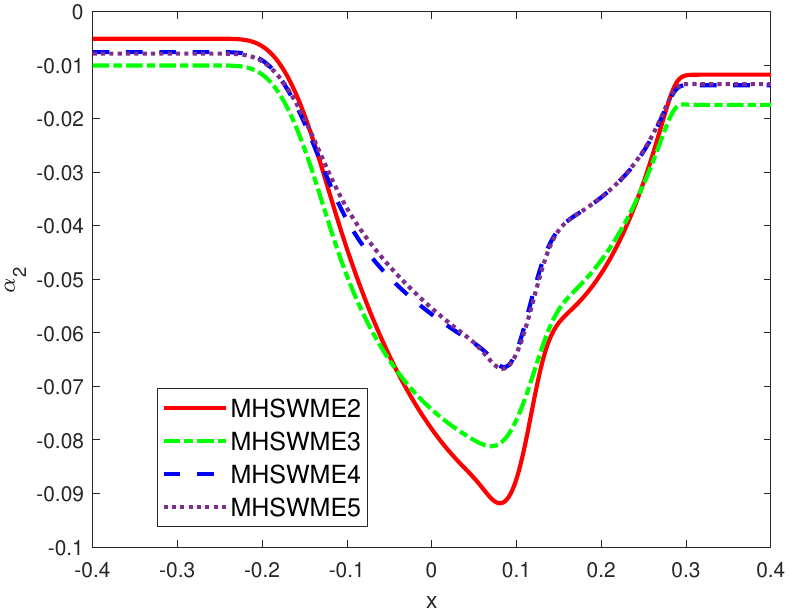}
        \caption{\textrm{MHSWME}}
        \label{fig:smalldam_MHSWME_alpha2}
    \end{subfigure}
    \vskip\baselineskip
    \begin{subfigure}[b]{0.47\textwidth}
        \centering
        \includegraphics[width=\textwidth]{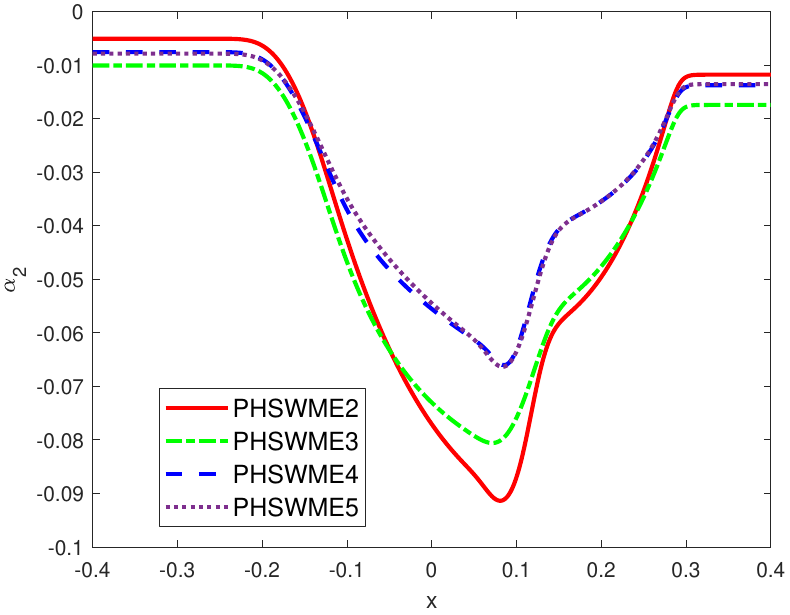}
        \caption{PHSWME}
        \label{fig:smalldam_PHSWME_alpha2}
    \end{subfigure}
    \hfill
    \begin{subfigure}[b]{0.47\textwidth}
        \centering
        \includegraphics[width=\textwidth]{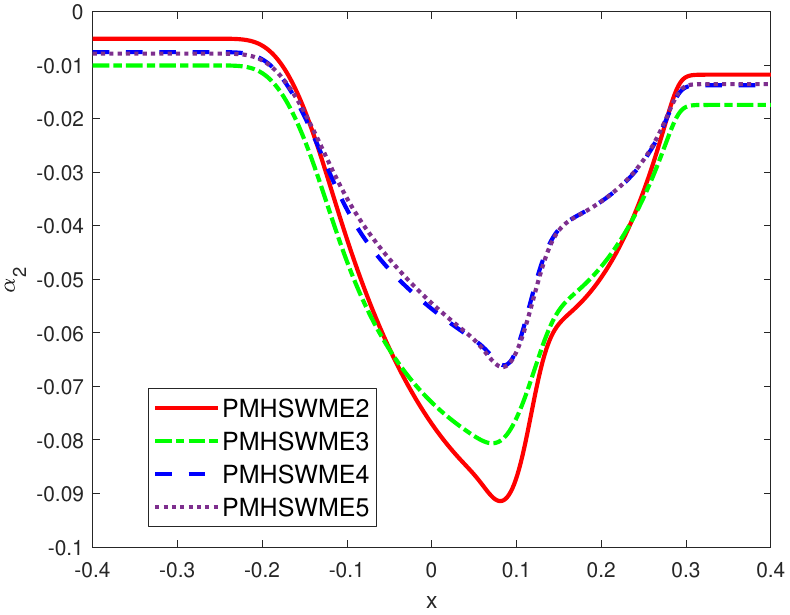}
        \caption{PMHSWME}
        \label{fig:smalldam_PMHSWME_alpha2}
    \end{subfigure}
    \caption{Dam break test case quadratic coefficient $\alpha_2$ solutions for SWME (top left), HSWME (top right), SWLME (middle left), MHSWME (middle right), PSWME (bottom left), PMSWME (bottom right), for $N=1,2,3,4,5$. Note that all models are equivalent for $N=1$, so this is only shown in the SWME plot. SWME5 is unstable and left out.}
    \label{fig:sol_dambreak_alpha2}
\end{figure}

\appendix
\begin{section}{Appendix: Notation}
Table \ref{tab:variables_formulas} summarizes the main notation for variables, coefficients used in this work.
\begin{table}[h]
\centering
\caption{Notation used throughout this work.}
\label{tab:variables_formulas}
\begin{tabular}{|l|l|l|}
\hline
\textbf{Notation} & \textbf{Description} & \textbf{Equation} \\ \hline
$h$ & water height & \eqref{eq:p_vars} \\ \hline
$u_m$ & mean velocity & \eqref{eq:p_vars} \\ \hline
$\alpha_i$ & moment coefficients & \eqref{eq:p_vars} \\ \hline
$U_{p}$ & primitive variables & \eqref{eq:p_vars} \\ \hline
$A^{\textrm{model}}_{p}$ & primitive system matrix of model & e.g. \eqref{eq:p} \\ \hline
$U_{c}$ & convective variables & \eqref{eq:c_vars} \\ \hline
$A^{\textrm{model}}_{c}$ & convective system matrix of model & e.g. \eqref{eq:c} \\ \hline
$T(U_p)$ & transformation from primitive to convective variables & \eqref{eq:dervars}, \eqref{eq:inv_dervars} \\ \hline
$A_{i j k}$, $B_{i j k}$ & SWME coefficients & \eqref{eq:Aijk_Bijk} \\ \hline
$\mathcal{A}_{i,l}$ & lower right block of SWME convective system matrix & \eqref{eq:A_block} \\ \hline
$\chi_c^{\textrm{model}}(\lambda)$ & characteristic polynomial of model system matrix & e.g. \eqref{eq:chi_HSWME_c} \\ \hline
$A_2$ & submatrix appearing in characteristic polynomial & \eqref{eq:A_2} \\ \hline
$a_i$, $c_i$ & entries of submatrix appearing in characteristic polynomial & \eqref{eq:A_2} \\ \hline
$\chi_{A_2}(\lambda)$ & characteristic polynomial of submatrix $A_2$ & \eqref{eq:A_2_char_poly} \\ \hline
$P'_{N+1}$ & derivative of the $(N+1)$-degree Legendre polynomial & e.g. \eqref{eq:A_2_char_poly} \\ \hline
$q_i$ & associated polynomial sequence used in Lemma 3 & \eqref{eq:q_sequence1}, \eqref{eq:q_sequence2} \\ \hline
$\lambda_i$ & eigenvalues & e.g. \eqref{eq:HSWME_c_EV} \\ \hline
$\Fr$ & Froude number & e.g. \eqref{eq:SWLME_steady_states} \\ \hline
$\Ma_i$ & dimensionless $i$-th moment number & e.g. \eqref{eq:SWLME_steady_states} \\ \hline
\end{tabular}
\end{table}

\end{section}

\bibliographystyle{abbrv}
\bibliography{references}

\end{document}